\documentclass[11pt,twoside,reqno]{amsart}

\usepackage[utf8]{inputenc}
\usepackage[T1]{fontenc}
\usepackage{lmodern}
\usepackage{amsmath,amsfonts,amsthm,amssymb}
\usepackage{mathrsfs}
\usepackage{graphicx,psfrag,overpic}
\graphicspath{{figures/}}
\usepackage{bm,enumerate}
\usepackage{cases}
\usepackage[all,cmtip,line]{xy}
\usepackage{array,CJK}
\usepackage{setspace}
\usepackage{color}
\definecolor{orange}{rgb}{1,0.5,0}
\usepackage{yhmath}

\usepackage{tikz}
\usetikzlibrary{arrows.meta,calc}
\usepackage{booktabs}

\numberwithin{equation}{section}
\numberwithin{figure}{section}

\usepackage{graphics}
\usepackage{epsfig}
\newtheorem{claim}{\bf \t}[part]

\newtheorem{theorem}{Theorem}[section]

\newtheorem{lemma}{Lemma}[section]
\newtheorem{proposition}[theorem]{Proposition}
\newtheorem{remark}{Remark}[section]

\newtheorem{definition}{Definition}[section]
\newtheorem{thm}{Theorem}[section]

\usepackage[hidelinks]{hyperref}

\hypersetup{pdftitle={Reconstruction of a Wedge and Post-Shock Flow from a Prescribed Leading Shock},
pdfauthor={Qianfeng Li and Meizhi Qian}}
\title[Inverse reconstruction from prescribed shocks]{Reconstruction of a Wedge and Post-Shock Flow from a Prescribed Leading Shock}

\author{Qianfeng Li$^{1}$ and Meizhi Qian$^{2}$}

\thanks{$^{1}$ Department of Mathematics, Friedrich-Alexander-Universität Erlangen-Nürnberg, Cauerstr. 11, 91058 Erlangen, Germany. Email: \href{mailto:qianfeng.li@fau.de}{\texttt{qianfeng.li@fau.de}}}

\thanks{$^{2}$ Department of Mathematics, Friedrich-Alexander-Universität Erlangen-Nürnberg, Cauerstr. 11, 91058 Erlangen, Germany. Email: \href{mailto:meizhi.qian@fau.de}{\texttt{meizhi.qian@fau.de}}}

\makeatletter
\@ifundefined{subjclassname@2020}{\@namedef{subjclassname@2020}{\textup{2020} Mathematics Subject Classification}}{}
\makeatother

\begin{document}
	\begin{abstract}
        In this paper, we reconstruct a wedge and its post-shock flow from a prescribed leading shock and a uniform supersonic incoming state. Full Euler Rankine--Hugoniot relations supply the shock data, while a reduced interior model combines homogeneous acoustic characteristic equations, entropy transport, and the Bernoulli relation. A hodograph transformation yields a linear hyperbolic equation for the physical coordinate. Comparing a shock-to-wedge mass bound with a hodograph non-degeneracy threshold gives a sufficient, data-dependent reconstruction criterion. Under the stated assumptions, this criterion ensures a unique \(C^1\) post-shock flow and a \(C^2\) wedge throughout the characteristic domain determined by the prescribed shock segment. Numerical reconstructions examine empirical grid convergence, Mach-number dependence, and discrepancies between the reduced-model and full-Euler inverse reconstructions.
\end{abstract}

\maketitle
\section{Introduction}

Supersonic flow past wedges is a classical problem in gas dynamics \cite{Courant1948}. In the direct problem, the wedge shape and the incoming flow are prescribed, and one seeks the resulting flow field together with the leading shock; see \cite{ChenLi2008,ChenZhangZhu2006,Hu2013,Schaeffer1976,Yin2006,Zhang1999,Zhang2003} and references therein. In this paper, we consider the corresponding inverse problem: the leading shock and the uniform incoming flow are prescribed, while the corresponding wedge and the post-shock flow field are to be determined.

The two-dimensional steady full compressible Euler equations for a polytropic gas read
\begin{equation}\label{eq:fullEuler}
\left\{
\begin{aligned}
&(\rho u)_x+(\rho v)_y=0,\\&
(\rho u^2+p)_x+(\rho uv)_y=0,\\&
(\rho uv)_x+(\rho v^2+p)_y=0,\\&
(\rho uE+pu)_x+(\rho vE+pv)_y=0,
\end{aligned}
\right.
\end{equation}
where
\[
E=\frac{u^2+v^2}{2}+e,
\qquad
e=\frac{p}{(\gamma-1)\rho},
\qquad 1<\gamma<3.
\]
Here $\rho$, $(u,v)$, $p$, $E$, and $e$ denote the density, velocity, pressure, total energy per unit mass, and internal energy per unit mass, respectively. See Appendix~\ref{app:full-euler-characteristics} for the characteristic form of \eqref{eq:fullEuler}. The incoming flow is a uniform horizontal supersonic state, normalized as
\begin{equation}\label{eq:initialdata}
(\rho,u,v,p)|_{x=0}
=(\rho_\infty,1,0,p_\infty),
\end{equation}
with $\rho_\infty>0$, $p_\infty>0$, and
\[
M_\infty^2=\frac{\rho_\infty}{\gamma p_\infty}>1,
\qquad
\epsilon:=M_\infty^{-2}\in(0,1).
\]
On the wedge surface $\mathbf W:=\{(x,y):y=f(x),\ x\ge0,\ f(0)=0\},$ the slip boundary condition reads
\begin{equation}\label{eq:slip}
v=uf'(x).
\end{equation}
Across the leading shock $\mathbf S:=\{(x,y):y=\phi(x),\ x\ge0,\ \phi(0)=0\},$ Rankine--Hugoniot relations and the compressive entropy condition hold. The state ahead of the prescribed leading shock is the uniform incoming state. The full Euler Rankine--Hugoniot relations on the compressive branch determine the downstream trace \cite{Courant1948,NACA1135,Anderson2003}: \begin{equation}\label{eq:shockcondition}
    (\rho, u, v, p)\big|_{x>0,y=\phi(x)-}=(\rho_{\rm sh}, u_{\rm sh}, v_{\rm sh}, p_{\rm sh})(x)
\end{equation}
where
\begin{equation*}
\begin{aligned}
\rho_{\rm sh}(x)&=\rho_\infty r,
&
p_{\rm sh}(x)&=p_\infty\left[
1+\frac{2\gamma}{\gamma+1}(M_{n,\infty}^2-1)\right],\\
u_{\rm sh}(x)&=\frac{1+s^2/r}{1+s^2},
&
v_{\rm sh}(x)&=\frac{s(1-1/r)}{1+s^2},
\end{aligned}
\end{equation*}
with
\[
s=\phi'(x), \qquad r=\frac{(\gamma+1)M_{n,\infty}^2}
{(\gamma-1)M_{n,\infty}^2+2},
\qquad
M_{n,\infty}^2=\frac{s^2}{\epsilon(1+s^2)}>1.
\]

For a fixed $X_*>0$, consider the prescribed leading shock segment
\[
\mathbf S^*
:=\mathbf S\cap\{(x,y):0\le x\le X_*\}.
\]
Let $\mathbf\Gamma^*$ denote the backward $\lambda_+$-characteristic issuing from
$(X_*,\phi(X_*))$.
We denote by $\Omega^*$ the bounded domain enclosed by
$\mathbf S^*$, $\mathbf\Gamma^*$, and $\mathbf W$, and set
\(\mathbf W^*:=\partial\Omega^*\cap\overline{\mathbf W}\) for the corresponding
wedge segment. The inverse problem is to determine $\mathbf W^*$ and the
post-shock flow in $\Omega^*$.

For the inverse problem with a prescribed leading shock, Wang \cite{Wang2011} established global existence and uniqueness for planar supersonic potential flow under smallness assumptions on the variation of the shock slope and a weighted norm of its second derivative. Li and Zhang \cite{LiZhang2022} subsequently removed the smallness requirement on the shock-slope variation under suitable geometric assumptions. They used mass conservation to estimate characteristic lengths and obtained both a global existence theorem under a small weighted second-derivative condition and a criterion excluding global classical solutions, thereby clarifying the role of the prescribed shock geometry in the reconstruction. For conical configurations, Hu, Li, and Zhang \cite{HuLiZhang2024Cone} established global classical solvability for axisymmetric hypersonic potential flow by constructing a degenerate background solution and applying scaling and perturbation arguments. Zhang \cite{Zhang2024LocalCone} obtained a local reconstruction near the cone vertex for a small angle of attack and a sufficiently large incoming Mach number, together with asymptotic expansions for the velocity and the cone slope. Li and Wang \cite{LiWang2006Shock} studied a related inverse generalized Riemann problem for quasilinear hyperbolic conservation laws, recovering initial data from prescribed non-degenerate shock trajectories.

Computational methods have also been developed to reconstruct solid boundaries from prescribed shock geometries. Sobieczky, Dougherty, and Jones \cite{Sobieczky1990Waverider} studied hypersonic waverider design from given shock waves using osculating-cone constructions and cross-marching techniques. Subsequent characteristic-based approaches include the three-dimensional inverse method of Yu, Huang, and Xia \cite{YuHuangXia2020} for inlet-lip design, the curved-shock characteristic method of Shi et al.\ \cite{Shi2021} for planar and axisymmetric flowfields, and the space-streamline-based method of Zhou et al.\ \cite{ZhouJinZhuangShi2022} for three-dimensional inverse design. A different class of inverse problems prescribes the wall pressure rather than the shock position; see \cite{LiZhang2025BendingWall,PuZhang2023} for potential flow and \cite{ChenPuZhang2025} for the full Euler system.

A curved shock generally produces nonconstant entropy and vorticity in the downstream full Euler flow. Consequently, the acoustic characteristic equations contain entropy-gradient terms, which couple the acoustic variables to streamline transport. See, for example, \cite{Courant1948,ChenZhangZhu2006,ChenLi2008}{}. We study a reduced interior system that omits these acoustic source terms while retaining the full Euler jump relations on the prescribed shock. Precisely, the reduced model reads
\begin{equation}\label{eq:PotantialFlow}
\left\{
\begin{aligned}
&R_x+\lambda_-R_y=0,
\\&
S_x+\lambda_+S_y=0,
\\&
u\Sigma_x+v\Sigma_y=0,
\\&
\frac{u^2+v^2}{2}
+\frac{\gamma p}{(\gamma-1)\rho}
=
\frac12+\frac{\gamma p_\infty}{(\gamma-1)\rho_\infty},
\end{aligned}
\right.
\end{equation}
with Riemann variables $R,S$ defined in \eqref{eq:defofRiemanninvariants}, characteristic slopes $\lambda_\pm$ in \eqref{eq:defoffcharacteristics}, and entropy variable $\Sigma$ in \eqref{eq:defofflowstates}. The hybrid model also implies mass conservation; see Appendix~\ref{app:mass-conservation}.

In this article, we study the inverse problem by a hodograph transformation method. A key point is to prove that the hodograph transformation is locally invertible, that is, its Jacobian is nonzero. Under the shock assumptions below, the signs $x_R<0<x_S$ ensure that the Jacobian is nonzero on the image of the leading shock in the phase plane. Our analysis gives a quantitative condition preserving this sign structure during continuation toward the wedge. Mass conservation supplies a complementary upper bound for the shock-to-wedge distance. Comparing the two bounds yields a sufficient criterion for a classical reconstruction throughout the characteristic domain determined by the prescribed shock segment.

The numerical study examines the reconstructed wedge geometry, the hodograph sign conditions, empirical grid convergence, and Mach-number dependence. Comparisons with full-Euler inverse Cauchy reconstructions using identical shocks and upstream data assess the effect of the reduced interior model. The discrepancies persist over the tested Mach-number range, so the observed stabilization within the hybrid model does not establish convergence toward full Euler. These comparisons concern inverse reconstructions and do not constitute independent forward validation.

Define a compact region in the phase plane
\[
\mathcal P:=\bigl\{(R,S):R(X_*,\phi(X_*))\le R\le R(0,0),\quad
\chi(R)\le S\le S(X_*,\phi(X_*))\bigr\},
\]
where $S=\chi(R)$ represents the prescribed leading shock trace in the phase plane, and the values of $R$ and $S$ on the shock are taken from the post-shock side. Under the hypotheses of the theorem below, Lemma~\ref{lem:MonotonicityofRS} ensures that $\chi$ and $\mathcal P$ are well defined. Moreover, define the shock-to-wedge mass bound and the hodograph non-degeneracy threshold, respectively, by
\begin{equation}\label{eq:bounds-intro}
B_{\rm mass}(x):=
\frac{\rho_\infty\,\phi\!\left(x\right)}
{\sqrt{\gamma\rho p}(x,\phi(x))\sqrt{1+\beta(R_{\rm sh}(x))}}
\end{equation}
and
\begin{equation}\label{eq:nd-bound-intro}
B_{\rm nd}(x):=\frac{-x_R(R_{\rm sh}(x),S_{\rm sh}(x))}{\alpha(R_{\rm sh}(x))}.
\end{equation}
Here $(R_{\rm sh}(x),S_{\rm sh}(x))$ is the post-shock phase trace, and $\alpha(R)>0$ and $\beta(R)\ge0$ are defined from the characteristic coefficients on $\mathcal P$ in \eqref{eq:phase-extrema} and \eqref{eq:beta-definition}.

We use the following assumptions:
\begin{itemize}
    \item[ ] (H1): For $x\in[0,X_*]$, $\phi''(x)>0$, $\phi'(x)\in\mathcal I_{\rm sup}$, and $\mathcal D(\phi'(x))<0$.
    \item[ ] (H2): $M<+\infty$ and $-\frac{\pi}{2}<\theta-\theta_m<\theta+\theta_m<\frac{\pi}{2}$ in the compact region $\mathcal P$.
\end{itemize}
Here $\mathcal I_{\rm sup}$ and $\mathcal D$ are given in \eqref{eq:admissibleslope} and \eqref{eq:Dtdefinition}, respectively, and the flow variables $M, \theta_m, \theta$ are given in \eqref{eq:defofflowstates}.

We now state the main theorem.

\begin{thm}\label{thm:main}
Let $X_*>0$ and let $y=\phi(x)$, $0\le x\le X_*$, be a prescribed shock, where $\phi\in C^2([0,X_*])$ and $\phi(0)=0$. Suppose that \textup{(H1)--(H2)} hold.
If for $0\le x\le X_*,$
\begin{equation}\label{eq:maincriteria}
B_{\rm mass}(x)< B_{\rm nd}(x),
\end{equation}
then the hybrid problem \eqref{eq:slip}--\eqref{eq:PotantialFlow} admits a $C^2$ wedge surface $\mathbf W^*$ and a $C^1$ post-shock flow in $\Omega^*$.
\end{thm}

\begin{remark}
The assumptions on $\phi$ select a compressive shock with a supersonic post-shock state and yield the initial hodograph signs $x_R<0<x_S$. The condition $u>c>0$ ensures that both acoustic characteristic slopes are finite and strictly ordered. The inequality \eqref{eq:maincriteria} is the reconstruction criterion: it compares two bounds determined by the prescribed data. It is sufficient; necessity and sharpness are not asserted.
\end{remark}

The rest of this paper is organized as follows.
Section~\ref{sec:phase-reformulation} reformulates the hybrid inverse
problem in the phase plane through the hodograph transformation.
We derive the transformed equations, establish the properties of the
Cauchy data induced by the prescribed shock, and express the unknown
wedge as an integral curve in the phase plane.
Section~\ref{sec:proof-main} first establishes global existence and
uniqueness for the Cauchy problem associated with the linear hyperbolic
hodograph equation on $\mathcal P$. We then compare the
mass-conservation bound on the shock-to-wedge distance with the
hodograph non-degeneracy threshold to preserve invertibility in the
region between the wedge and the prescribed leading shock.
Together with the continuation of the wedge and the global injectivity
of the hodograph map, these estimates yield the physical reconstruction
and complete the proof of Theorem~\ref{thm:main}.

\section{Reformulation of the inverse problem in phase space}\label{sec:phase-reformulation}

In this section, we reformulate the inverse problem \eqref{eq:slip}--\eqref{eq:PotantialFlow} in the
phase plane. Write $s=\phi'(x)>0$ for the shock slope. Throughout this section, the subscript ``$\rm sh$'' denotes evaluation at $(x,\phi(x)-)$.

We use the flow variables
\begin{equation}\label{eq:defofflowstates}
\begin{aligned}
&\text{sonic speed }c:=\sqrt{\frac{\gamma p}{\rho}},
\quad
\text{flow speed }q:=\sqrt{u^2+v^2},
\quad
 \text{entropy variable }\Sigma:=\frac{p}{\rho^\gamma},\\
&\text{Mach number }M:=\frac{q}{c}, \quad \text{Mach angle }\theta_m:=\arcsin\frac{1}{M},
\quad
\text{flow angle }\theta:=\arctan\frac{v}{u},
\end{aligned}
\end{equation}
and the two acoustic characteristic slopes are
\begin{equation}\label{eq:defoffcharacteristics}
    \lambda_-:=\tan(\theta-\theta_m), \qquad \lambda_+:=\tan(\theta+\theta_m),
\end{equation}
and the corresponding Riemann variables are
\begin{equation}\label{eq:defofRiemanninvariants}
R:=\theta+\nu(M),
\qquad
S:=\theta-\nu(M),
\end{equation}
where the Prandtl--Meyer function is
\[
\nu(M)
=
\sqrt{\frac{\gamma+1}{\gamma-1}}
\arctan\!\left(
\sqrt{\frac{\gamma-1}{\gamma+1}(M^2-1)}
\right)
-
\arctan\sqrt{M^2-1}.
\]

Since $\nu$ is strictly increasing for $M>1$, the relations
\begin{equation}\label{eq:phase-acoustic-variables}
\theta=\frac{R+S}{2},\qquad
M=\nu^{-1}\!\left(\frac{R-S}{2}\right)
\end{equation}
recover the flow angle and Mach number whenever $(R-S)/2$ lies in the
supersonic range. Consequently, together with the last formula in \eqref{eq:PotantialFlow}, we have that $u,v,c$ and $\lambda_\pm$ are functions of $(R,S)$ alone.

Introduce the hodograph transformation
\[
T:(x,y)\longmapsto(R(x,y),S(x,y)).
\]
Where $J:=\partial(R,S)/\partial(x,y)\ne0$, write its local inverse as
$F(R,S)=(x(R,S),y(R,S))$. Moreover, the inverse hodograph Jacobian is
\begin{equation}\label{eq:inverse-hodograph-jacobian}
\frac{\partial(x,y)}{\partial(R,S)}
=x_Ry_S-x_Sy_R
=(\lambda_--\lambda_+)x_Rx_S.
\end{equation}

Under the hodograph transformation, the first two equations of
\eqref{eq:PotantialFlow} become
\begin{equation}\label{eq:hodographsystemphase}
y_R=\lambda_+x_R,
\qquad
y_S=\lambda_-x_S.
\end{equation}
The compatibility condition $(y_R)_S=(y_S)_R$ gives
\begin{equation}\label{eq:hodographxequationphase}
x_{RS}
=
\frac{\partial_R\lambda_-}{\lambda_+-\lambda_-}\,x_S
-
\frac{\partial_S\lambda_+}{\lambda_+-\lambda_-}\,x_R.
\end{equation}
Note that the coefficients depend only on
the phase coordinates $(R,S)$. \eqref{eq:hodographxequationphase} is a linear
hyperbolic equation for $x$. Once $x$ is known, the compatible first-order
relations \eqref{eq:hodographsystemphase} determine $y$ up to an additive
constant. Lines $R=\mathrm{constant}$ and $S=\mathrm{constant}$ map to
the $\lambda_-$- and $\lambda_+$-characteristics, respectively.

The following trace property allows us to represent the shock as a graph
in the phase plane.

\begin{lemma}[Monotonicity along a convex shock]\label{lem:MonotonicityofRS}
Let $\phi\in C^2([0,X_*])$ describe a strictly convex shock with positive
curvature, $\phi''(x)>0$, and suppose that
$\phi'(x)\in\mathcal I_{\rm sup}$ for all $x\in[0,X_*]$ with $\mathcal I_{\rm sup}$ in \eqref{eq:admissibleslope}
selecting the compressive, entropy-admissible, supersonic branch.
Then the post-shock speed and flow angle satisfy
\[
q'_{\rm sh}(x)<0,\qquad \theta'_{\rm sh}(x)>0,
\]
and the Riemann traces of \eqref{eq:defofRiemanninvariants} satisfy
\[
S'_{\rm sh}(x)>0,
\qquad
\operatorname{sgn}R'_{\rm sh}(x)
=\operatorname{sgn}\mathcal D(\phi'(x)),
\]
where primes denote differentiation with respect to \(x\), and $\mathcal D$ is the polynomial defined in
\eqref{eq:Dtdefinition}.
\end{lemma}

The proof, including the explicit polynomial $\mathcal D$, is given in
Appendix~\ref{app:sign-regime}.

If, in addition, $\mathcal D(\phi'(x))<0$ on $[0,X_*]$, the preceding
lemma gives
\begin{equation}\label{eq:monotonicityofRshSsh}
R'_{\rm sh}<0<S'_{\rm sh}.
\end{equation}
Hence the post-shock trace in the phase plane,
\[
\Gamma_{\rm sh}
:=\{(R_{\rm sh}(x),S_{\rm sh}(x)):0\le x\le X_*\},
\]
has the graph representations
\[
\Gamma_{\rm sh}
=\{(R,\chi(R)):R_{\min}\le R\le R_{\max}\}
=\{(\sigma(S),S):S_{\min}\le S\le S_{\max}\},
\]
where
\[
\begin{aligned}
R_{\min}&:=R_{\rm sh}(X_*),& R_{\max}&:=R_{\rm sh}(0),\\
S_{\min}&:=S_{\rm sh}(0),& S_{\max}&:=S_{\rm sh}(X_*).
\end{aligned}
\]
The functions $\chi$ and $\sigma$ are $C^1$, strictly decreasing, and
mutually inverse; in particular,
\[
\chi'(R_{\rm sh}(x))
=\frac{S'_{\rm sh}(x)}{R'_{\rm sh}(x)}<0.
\]

The compact phase region introduced in Section~1 can therefore be written as
\begin{equation}\label{eq:phase-domain}
\begin{aligned}
\mathcal P
&=\{(R,S):R_{\min}\le R\le R_{\max},\ \chi(R)\le S\le S_{\max}\}\\
&=\{(R,S):S_{\min}\le S\le S_{\max},\ \sigma(S)\le R\le R_{\max}\}.
\end{aligned}
\end{equation}
See Figure \ref{fig:phase-domains}. It is determined entirely by the prescribed shock and incoming state.

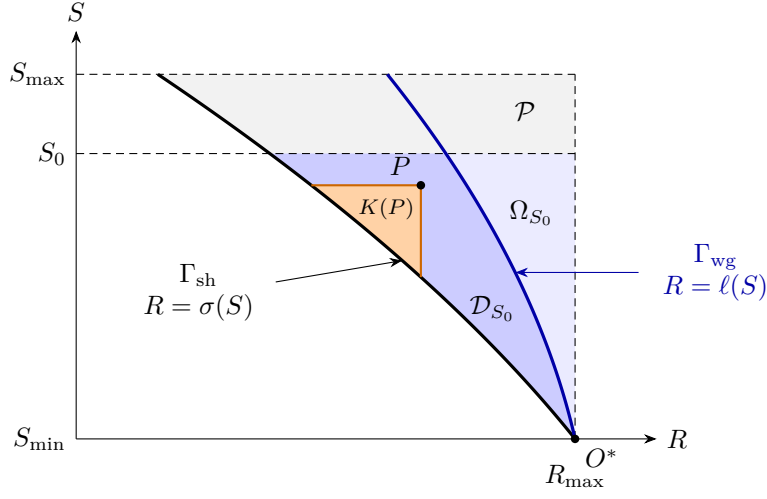
\begin{figure}[htbp]
\centering
\begin{tikzpicture}[x=1.2cm,y=1.05cm,>=Stealth,font=\small,
  declare function={
    phaseShock(\t)=5.5-0.7*\t-(0.3/4.6)*\t*\t;
    phaseWedge(\t)=5.5-0.2*\t-(0.25/4.6)*\t*\t;
  }]
\pgfmathsetmacro{\phasePR}{3.8}
\pgfmathsetmacro{\phasePS}{3.2}
\pgfmathsetmacro{\phaseChiP}{(-0.7+sqrt(0.7*0.7+4*(0.3/4.6)*(5.5-\phasePR)))/(2*(0.3/4.6))}
\fill[black!5]
  plot[domain=0:4.6,variable=\t,samples=61] ({phaseShock(\t)},\t)
  --(5.5,4.6)--cycle;
\fill[blue!8]
  plot[domain=0:3.6,variable=\t,samples=49] ({phaseShock(\t)},\t)
  --(5.5,3.6)--cycle;
\fill[blue!20]
  plot[domain=0:3.6,variable=\t,samples=49] ({phaseShock(\t)},\t)
  --({phaseWedge(3.6)},3.6)
  plot[domain=3.6:0,variable=\t,samples=49] ({phaseWedge(\t)},\t)
  --cycle;
\fill[orange!35]
  plot[domain=\phaseChiP:\phasePS,variable=\t,samples=25] ({phaseShock(\t)},\t)
  --(\phasePR,\phasePS)--cycle;
\draw[->] (0,0)--(6.4,0) node[right] {$R$};
\draw[->] (0,0)--(0,5.15) node[above] {$S$};
\draw[densely dashed] (0,4.6)--(5.5,4.6);
\draw[densely dashed] (0,3.6)--(5.5,3.6);
\draw[densely dashed] (5.5,0)--(5.5,4.6);
\draw[very thick]
  plot[domain=0:4.6,variable=\t,samples=61] ({phaseShock(\t)},\t);
\draw[very thick,blue!65!black]
  plot[domain=0:4.6,variable=\t,samples=61] ({phaseWedge(\t)},\t);
\draw[thick,orange!80!black]
  (\phasePR,\phaseChiP)--(\phasePR,\phasePS)--({phaseShock(\phasePS)},\phasePS);
\fill (\phasePR,\phasePS) circle (1.6pt) node[above left] {$P$};
\fill (5.5,0) circle (1.6pt) node[below right] {$O^*$};
\node[left] at (0,4.6) {$S_{\max}$};
\node[left] at (0,3.6) {$S_0$};
\node[left] at (0,0) {$S_{\min}$};
\node[below] at (5.5,-0.2) {$R_{\max}$};
\node[align=center] at (1.35,1.85)
{$\Gamma_{\rm sh}$\\$R=\sigma(S)$};
\draw[->,thin] (2.2,1.98)--({phaseShock(2.25)},2.25);
\node[align=center,blue!65!black] at (7.05,2.1)
{$\Gamma_{\rm wg}$\\$R=\ell(S)$};
\draw[->,thin,blue!65!black] (6.2,2.1)--({phaseWedge(2.1)},2.1);
\node[font=\scriptsize] at (3.42,2.91) {$K(P)$};
\node at (4.59,1.65) {$\mathcal D_{S_0}$};
\node at (5.02,2.85) {$\Omega_{S_0}$};
\node at (4.95,4.14) {$\mathcal P$};
\end{tikzpicture}
\caption{Schematic of the phase regions with curved shock and wedge
boundaries. The characteristic triangle $K(P)$ lies below and to the
left of $P$, between the shock and its horizontal and vertical
sections. The shock-to-wedge region $\mathcal D_{S_0}$ is a subset of
$\Omega_{S_0}$.}
\label{fig:phase-domains}
\end{figure}

The physical position of each shock point prescribes the Cauchy data
\begin{equation}\label{eq:phase-shock-data}
(x,y)(R_{\rm sh}(\xi),S_{\rm sh}(\xi))
=(\xi,\phi(\xi)),\qquad 0\le\xi\le X_*.
\end{equation}
Their tangential derivatives, together with
\eqref{eq:hodographsystemphase}, determine the derivatives of $x$ over $\Gamma_{\rm sh}$.

\begin{lemma}[Initial hodograph derivatives]\label{lem:initialsignsphase}
Assume \textup{(H1)}--\textup{(H2)}. Then
\begin{equation}\label{eq:xRxSinitialphase}
\left.x_R\right|_{\Gamma_{\rm sh}}
=\frac{\phi'-\lambda_-}
{(\lambda_+-\lambda_-)R'_{\rm sh}}<0,
\qquad
\left.x_S\right|_{\Gamma_{\rm sh}}
=\frac{\lambda_+-\phi'}
{(\lambda_+-\lambda_-)S'_{\rm sh}}>0.
\end{equation}
\end{lemma}

\begin{proof}
Differentiating \eqref{eq:phase-shock-data} with respect to $\xi$ and
using \eqref{eq:hodographsystemphase}, we obtain
\[
x_RR'_{\rm sh}+x_SS'_{\rm sh}=1,
\qquad
\lambda_+x_RR'_{\rm sh}+\lambda_-x_SS'_{\rm sh}=\phi'(\xi).
\]
Here the trace derivatives and shock quantities are evaluated at $\xi$.
Solving this system yields the formulas in
\eqref{eq:xRxSinitialphase}. For the compressive full Euler shock in the
forward characteristic regime, the downstream characteristic slopes obey
\begin{equation}\label{eq:differenceofthespeed}
\lambda_-<\phi'<\lambda_+
\qquad\text{on }\Gamma_{\rm sh}.
\end{equation}
Indeed, if $\beta=\arctan\phi'$ is the shock angle, conservation of
tangential velocity and compression give $\beta>\theta_{\rm sh}$, while
the downstream normal Mach number is below one. Therefore
$0<\beta-\theta_{\rm sh}<\theta_{m,\rm sh}$; the forward characteristic
regime allows us to take tangents without crossing a vertical direction.
Together with \eqref{eq:monotonicityofRshSsh}, this proves the asserted signs.
\end{proof}

To complete the interior formulation, we transform the entropy equation
in \eqref{eq:PotantialFlow}. Using \eqref{eq:hodographsystemphase}, it becomes
\begin{equation}\label{eq:entropyRS}
(v-u\lambda_-)x_S\Sigma_R
+
(u\lambda_+-v)x_R\Sigma_S
=0,
\end{equation}
with the prescribed values
\begin{equation}\label{eq:phase-entropy-data}
\Sigma(R_{\rm sh}(\xi),S_{\rm sh}(\xi))
=\Sigma_{\rm sh}(\xi)
:=\frac{p_{\rm sh}(\xi)}{\rho_{\rm sh}(\xi)^\gamma}.
\end{equation}
For a given pair \((x,y)\), this is a linear transport problem for
$\Sigma$. Once a positive $\Sigma$ is determined, the remaining state
variables are recovered by
\begin{equation}\label{eq:phase-thermodynamic-recovery}
\rho=\left(\frac{c^2}{\gamma\Sigma}\right)^{1/(\gamma-1)},
\qquad p=\Sigma\rho^\gamma.
\end{equation}

The hybrid model also satisfies mass conservation, as shown in
Lemma~\ref{lem:continuity-hybrid} in Appendix~\ref{app:mass-conservation}.
Using \eqref{eq:hodographsystemphase}, we rewrite it in the phase plane as
\begin{equation}\label{eq:massRS}
\partial_R\!\left[\rho(v-u\lambda_-)x_S\right]
+
\partial_S\!\left[\rho(u\lambda_+-v)x_R\right]
=0.
\end{equation}

It remains to express the unknown wedge in the $(R,S)$ coordinates.
Its phase image $\Gamma_{\rm wg}$ issues from
\[
O^*:=(R_{\max},S_{\min})=(R_{\rm sh}(0),S_{\rm sh}(0)).
\]
The slip condition is $dy=\tan\theta\,dx$ along this curve. Substituting
\eqref{eq:hodographsystemphase} yields
\begin{equation}\label{eq:presentationofWedgeinPhaseplane}
(\tan\theta-\lambda_-)x_S\,dS
-
(\lambda_+-\tan\theta)x_R\,dR
=0.
\end{equation}
On the branch $x_R<0<x_S$, the ordering
$\lambda_-<\tan\theta<\lambda_+$ allows us to write
$\Gamma_{\rm wg}$ as $R=\ell(S)$, where
\begin{equation}\label{eq:phase-wedge-ode}
\ell'(S)=
\left.\frac{(\tan\theta-\lambda_-)x_S}
{(\lambda_+-\tan\theta)x_R}\right|_{(\ell(S),S)}<0,
\qquad \ell(S_{\min})=R_{\max}.
\end{equation}
For any phase level $S_0\in(S_{\min},S_{\max}]$ reached by this curve,
the region between the shock and the wedge is
\begin{equation}\label{eq:DS0definition}
\mathcal D_{S_0}
:=
\{(R,S):\sigma(S)<R<\ell(S),\ S_{\min}<S<S_0\}.
\end{equation}
See Figure \ref{fig:phase-domains}. The required ordering is
$\sigma(S)<\ell(S)<R_{\max}$ for $S_{\min}<S\le S_0$.
At $S_0=S_{\max}$, the upper edge
\[
\{(R,S_{\max}):R_{\min}\le R\le\ell(S_{\max})\}
\]
maps to the terminal $\lambda_+$-characteristic issuing from the prescribed
shock endpoint. This terminal edge is characteristic and carries no
additional boundary data. Along the wedge,
\[
\frac{d}{dS}x(\ell(S),S)
=\left.x_S\frac{\lambda_+-\lambda_-}
{\lambda_+-\tan\theta}\right|_{(\ell(S),S)}>0,
\]
so its physical image is a graph $y=f(x)$ whenever the hodograph map
remains non-degenerate. The wedge is therefore determined by
\eqref{eq:phase-wedge-ode} once the acoustic Cauchy problem has been solved.

In summary, solving the original hybrid problem
\eqref{eq:slip}--\eqref{eq:PotantialFlow} in the physical plane is locally
equivalent, through a non-degenerate hodograph transformation, to solving
the hodograph system \eqref{eq:hodographsystemphase}, the entropy equation
\eqref{eq:entropyRS}, and the wedge equation \eqref{eq:phase-wedge-ode}
in the phase plane, with the prescribed shock Cauchy data
\eqref{eq:phase-shock-data} and \eqref{eq:phase-entropy-data}.
By \eqref{eq:inverse-hodograph-jacobian} and assumption \textup{(H2)},
the key issue is to show that the signs $x_R<0<x_S$ established on the
shock are preserved throughout $\overline{\mathcal D_{S_{\max}}}$, so that
the inverse hodograph Jacobian remains nonzero. We carry out this
analysis and complete the physical reconstruction in
Section~\ref{sec:proof-main}.

\section{Proof of the main theorem}\label{sec:proof-main}

In this section, we first solve the linear hodograph
Cauchy problem and establish a quantitative condition preserving its
non-degeneracy. After constructing the local wedge and transported
entropy, we derive the shock-to-wedge mass bound and use it to continue
the reconstruction to the terminal characteristic.

For the coefficients of \eqref{eq:hodographxequationphase}, write
\begin{equation}\label{eq:ABdefinitionphase}
a(R,S):=\frac{\partial_R\lambda_-}{\lambda_+-\lambda_-},
\qquad
b(R,S):=\frac{\partial_S\lambda_+}{\lambda_+-\lambda_-}.
\end{equation}
Differentiating \eqref{eq:phase-acoustic-variables} gives
\begin{equation}\label{eq:firstderivativeslambda}
\partial_R\lambda_-=
\frac{(1+\lambda_-^2)(\gamma+1)M^2}{4(M^2-1)},\qquad
\partial_S\lambda_+=
\frac{(1+\lambda_+^2)(\gamma+1)M^2}{4(M^2-1)}.
\end{equation}
Both expressions are positive for $M>1$. Together with \textup{(H2)}, this yields
\begin{equation}\label{eq:ABpositivephase}
\lambda_+-\lambda_->0,\qquad a>0,\qquad b>0
\quad\hbox{on }\mathcal P.
\end{equation}
All three functions are smooth on a neighborhood of the compact set
$\mathcal P$. We first construct the hodograph pair on this whole set.

Write
\[
g(R):=x_R(R,\chi(R)),\qquad
h(S):=x_S(\sigma(S),S),
\]
where the right-hand sides are the prescribed expressions in
\eqref{eq:xRxSinitialphase}. They are continuous, with $g<0<h$.

\begin{lemma}[Hodograph Cauchy problem]
\label{lem:local-hodograph-cauchy}
Assume \textup{(H1)--(H2)}. Then the problem
\eqref{eq:hodographsystemphase} with the shock data
\eqref{eq:phase-shock-data}--\eqref{eq:xRxSinitialphase} admits a unique pair
$(x,y)\in C^1(\mathcal P)$.
\end{lemma}

\begin{proof}
Introduce
\[p=x_R, \quad z=x_S.\]  The problem
\eqref{eq:hodographsystemphase} with the shock data
\eqref{eq:phase-shock-data}--\eqref{eq:xRxSinitialphase} is equivalent to the integral system
\begin{equation}\label{eq:hodograph-volterra}
\begin{aligned}
p(R,S)&=g(R)+\int_{\chi(R)}^S
\bigl[a(R,\eta)z(R,\eta)-b(R,\eta)p(R,\eta)\bigr]\,d\eta,\\
z(R,S)&=h(S)+\int_{\sigma(S)}^R
\bigl[a(r,S)z(r,S)-b(r,S)p(r,S)\bigr]\,dr.
\end{aligned}
\end{equation}

Let $X=C(\mathcal P)\times C(\mathcal P)$ and define $\mathcal TU$ by the right-hand sides
of \eqref{eq:hodograph-volterra}. The continuity of $a,b,g,h,\chi,\sigma$
makes $\mathcal T$ a self-map of $X$. Set
$K=\|a\|_\infty+\|b\|_\infty>0$ and $\lambda=2K$,
and equip $X$ with the weighted norm
\[
\|U\|_\lambda
:=\sup_{(R,S)\in\mathcal P}e^{-\lambda(R+S)}
\max\{|p(R,S)|,|z(R,S)|\},\qquad U=(p,z).
\]
Since $\mathcal P$ is compact, this norm is equivalent to the supremum
norm, so $X$ remains complete. For $U,V\in X$,
\[
\begin{aligned}
e^{-\lambda(R+S)}|(\mathcal TU-\mathcal TV)_1(R,S)|
&\le K\|U-V\|_\lambda
\int_{\chi(R)}^S e^{-\lambda(S-\eta)}\,d\eta\\
&\le \frac12\|U-V\|_\lambda.
\end{aligned}
\]
The horizontal integral satisfies the same bound, with the exponential
$e^{-\lambda(R-r)}$. Hence
$\|\mathcal TU-\mathcal TV\|_\lambda\le\frac12\|U-V\|_\lambda$.
The Banach fixed-point theorem therefore gives a unique
$(p,z)\in C(\mathcal P)\times C(\mathcal P)$ solving \eqref{eq:hodograph-volterra}
on the whole of $\mathcal P$.

Moreover, the integral equations \eqref{eq:hodograph-volterra} imply
\[
p_S=z_R=az-bp.
\]
Together with \eqref{eq:ABdefinitionphase}, this shows that
$p\,dR+z\,dS$ and $\lambda_+p\,dR+\lambda_-z\,dS$ are closed.
Since $\mathcal P$ is simply connected, they admit $C^1$ primitives
$x,y$ satisfying $x_R=p$, $x_S=z$, and \eqref{eq:hodographsystemphase}.
With $(x,y)(O^*)=(0,0)$, the initial derivative formulas
\eqref{eq:xRxSinitialphase} imply the prescribed shock data
\eqref{eq:phase-shock-data}. Indeed, the chain rule and \eqref{eq:xRxSinitialphase} give
\[
\begin{aligned}
\frac{d}{d\xi}x(R_{\rm sh}(\xi),S_{\rm sh}(\xi))
&=pR'_{\rm sh}+zS'_{\rm sh}=1,\\
\frac{d}{d\xi}y(R_{\rm sh}(\xi),S_{\rm sh}(\xi))
&=\lambda_+pR'_{\rm sh}+\lambda_-zS'_{\rm sh}=\phi'(\xi).
\end{aligned}
\]
Integrating from $\xi=0$, where both the reconstructed and prescribed positions are $(0,0)$, yields \eqref{eq:phase-shock-data}.
\end{proof}

The Cauchy problem is thus solved independently of invertibility away
from the shock. Since $a,b>0$, the equation also gives
\begin{equation}\label{eq:hodograph-cross-monotonicity}
x_{RS}=a x_S-b x_R>0
\qquad\text{wherever }x_R<0<x_S.
\end{equation}
In this region, $x_S$ increases with $R$ from a positive value on the shock,
while $x_R$ increases with $S$ from a negative value on the shock. To prevent
the latter derivative from reaching zero, we use an integrating factor.
Define the integrating-factor weight over $\mathcal{P}$,
\begin{equation}\label{eq:phase-weight}
\mathcal A(R,S):=
\exp\!\left(\int_{\chi(R)}^S b(R,\tau)\,d\tau\right)a(R,S)
\end{equation}
and its maximum on each vertical phase section,
\begin{equation}\label{eq:phase-extrema}
\alpha(R):=\max_{\chi(R)\le S\le S_{\max}}\mathcal A(R,S).
\end{equation}
By (H1)--(H2), the function $\mathcal A$ is continuous and positive on $\mathcal P$. Thus, $\alpha$ is continuous and positive on
$[R_{\min},R_{\max}]$.

We now express the bounds $B_{\rm mass}$ and $B_{\rm nd}$ in \eqref{eq:bounds-intro} and \eqref{eq:nd-bound-intro} as functions of the shock phase coordinate $R$. Keeping the notation, set
\[
\xi_{\rm sh}(R):=R_{\rm sh}^{-1}(R)=x(R,\chi(R)).
\]
The non-degeneracy threshold then reads
\begin{equation}\label{eq:phase-nd-bound}
B_{\rm nd}(R):=\frac{-g(R)}{\alpha(R)}>0.
\end{equation}
For $S_{\min}<S_0\le S_{\max}$, let
\begin{equation}\label{eq:Omegasigmadefinition}
\Omega_{S_0}:=
\{(R,S):S_{\min}<S<S_0,\ \sigma(S)<R<R_{\max}\}.
\end{equation}

For $P=(R,S)\in\Omega_{S_0}$, define its closed characteristic
triangle by
\[
K(P):=\{(r,\eta):\sigma(S)\le r\le R,\quad
\chi(r)\le\eta\le S\}.
\]
The vertical and horizontal segments in
\eqref{eq:hodograph-volterra}, together with all their preceding
segments, are contained in $K(P)$. Introduce the distance condition
\begin{equation}\label{eq:xdistanceconditionphase}
x(R,S)-x(R,\chi(R))< B_{\rm nd}(R)
\end{equation}
and the region on which this condition holds throughout each such triangle,
\begin{equation}\label{eq:hereditary-region}
\Omega^*_{S_0}:=
\{P\in\overline{\Omega_{S_0}}:\eqref{eq:xdistanceconditionphase}
\text{ holds at every point of }K(P)\}.
\end{equation}
See Figure~\ref{fig:phase-domains}.
% At the shock the condition holds automatically, since $B_{\rm nd}>0$.
% This definition ensures that every point of $K(P)$ off the shock has
% the same preceding-section property.

\begin{proposition}[Preservation of the hodograph derivative signs]
\label{prop:nonvanishingxRxS}
Under \textup{(H1)--(H2)}, the solution of
Lemma~\ref{lem:local-hodograph-cauchy} satisfies
$x_R<0<x_S$ in $\Omega^*_{S_{\max}}$.
\end{proposition}

\begin{proof}
Fix $P\in\Omega^*_{S_{\max}}$. Suppose, for contradiction, that
\[
E:=\{(r,s)\in K(P):x_R(r,s)\ge0\ \text{or}\ x_S(r,s)\le0\}
\ne\varnothing.
\]
By the continuity of $x_R$ and $x_S$, $E$ is compact. Choose any $Q=(R^q,S^q)\in E$ minimizing
$r+s$ over $E$. By (H1)--(H2), Lemma~\ref{lem:initialsignsphase} implies that $Q$ lies off the
shock. Both vertical and horizontal segments from the shock $\Gamma_{\rm sh}$ to $Q$ lie in
$K(P)$, and every point on either segment other than $Q$ has coordinate
sum less than $R^q+S^q$. By minimality, none of these points belongs to
$E$, so $x_R<0<x_S$ there.

On the one hand, at $Q$, continuity gives $x_R\le0\le x_S$;
since $Q\in E$, at least one derivative vanishes.

On the other hand, by \eqref{eq:hodographxequationphase}, we have
\begin{equation}\label{eq:exponentialform}
\begin{aligned}
\partial_S\left(x_R(R^q,S)e^{\int_{\chi(R^q)}^S b(R^q,\tau)\,d\tau}\right)
&=a(R^q,S)x_S(R^q,S)e^{\int_{\chi(R^q)}^S b(R^q,\tau)\,d\tau},\\
\partial_R\left(x_S(R,S^q)e^{-\int_{\sigma(S^q)}^R a(\tau,S^q)\,d\tau}\right)
&=-b(R,S^q)x_R(R,S^q)e^{-\int_{\sigma(S^q)}^R a(\tau,S^q)\,d\tau}.
\end{aligned}
\end{equation}
Integrating along the vertical and horizontal segments
\[
\{(R^q,\eta):\chi(R^q)\le\eta\le S^q\},\qquad
\{(r,S^q):\sigma(S^q)\le r\le R^q\},
\]
respectively, and using Lemma \ref{lem:initialsignsphase} together with $x_S\ge0$ on
this vertical segment and $x_R\le0$ on this horizontal segment, we obtain
\begin{equation}
    \begin{aligned}
        e^{\int_{\chi(R^q)}^{S^q} b(R^q,\tau)\,d\tau}x_R(Q)
&=g(R^q)+\int_{\chi(R^q)}^{S^q}\mathcal A(R^q,\eta)x_S(R^q,\eta)\,d\eta\\
&\le g(R^q)+\alpha(R^q)[x(Q)-x(R^q,\chi(R^q))],
\label{eq:vertical-factor-estimate}
    \end{aligned}
\end{equation}
and
\begin{equation}
    \begin{aligned}
e^{-\int_{\sigma(S^q)}^{R^q} a(r,S^q)\,dr}x_S(Q)
&=h(S^q)+\int_{\sigma(S^q)}^{R^q}
e^{-\int_{\sigma(S^q)}^r a(t,S^q)\,dt}[-b(r,S^q)x_R(r,S^q)]\,dr\\
&\ge h(S^q).
\label{eq:horizontal-factor-estimate}
\end{aligned}
\end{equation}
Furthermore, by $Q\in K(P)$ and  $P\in \Omega_{S_{\max}}^*$,  substituting \eqref{eq:xdistanceconditionphase} into  \eqref{eq:vertical-factor-estimate}, we obtain $x_R(Q)<0$, since $\alpha(R^q)>0$ and
\[
g(R^q)+\alpha(R^q)[x(Q)-x(R^q,\chi(R^q))]
<g(R^q)+\alpha(R^q)B_{\rm nd}(R^q)=0.
\] By Lemma \ref{lem:initialsignsphase}, \eqref{eq:horizontal-factor-estimate} implies $x_S(Q)>0$ directly. This contradicts the vanishing
of at least one derivative at $Q$. Hence $E$ is empty, which completes the proof.
\end{proof}

Under assumptions (H1) and (H2), by \eqref{eq:inverse-hodograph-jacobian}, the preceding proposition
ensures local invertibility in $\Omega^*_{S_0}$. To prove
Theorem~\ref{thm:main}, we next show that the wedge curve $R=\ell(S)$
in the phase plane can be extended over the entire interval
$[S_{\min},S_{\max}]$, with the corresponding flow region satisfying
$\mathcal D_{S_{\max}}\subset\Omega^*_{S_{\max}}$.

To this end, we use the following continuation property.

\begin{definition}[Admissible reconstruction up to a phase level]
\label{def:continuation-property}
For $S_{\min}<S_0\le S_{\max}$, the property $\mathsf H(S_0)$ means
that there is a unique function $\ell\in C^1([S_{\min},S_0])$ whose graph
solves \eqref{eq:presentationofWedgeinPhaseplane}, with
\[
\ell(S_{\min})=R_{\max},\qquad
\sigma(S)<\ell(S)<R_{\max},\quad \ell'(S)<0
\quad(S_{\min}<S\le S_0),
\]
and such that $x_R<0<x_S$ on ${\mathcal D_{S_0}}$, with
$\mathcal D_{S_0}$ as in \eqref{eq:DS0definition}.
We write $S=\omega(R)$ for the inverse of $R=\ell(S)$, defined on
$[\ell(S_0),R_{\max}]$.
\end{definition}

\begin{lemma}\label{lem:localH(S_0)}
Under \textup{(H1)--(H2)},
\eqref{eq:phase-wedge-ode} admits a unique solution
$\ell\in C^1([S_{\min},S_0])$  for some $S_{\min}<S_0\le S_{\max}$. Moreover, $\mathsf H(S_0)$ holds.
\end{lemma}

\begin{proof}
By Lemma~\ref{lem:local-hodograph-cauchy}, $(x,y)\in C^1(\mathcal P)$
and $x_{SR}$ is continuous. The strict shock signs $x_R<0<x_S$
persist near $O^*$, where the right-hand side of
\eqref{eq:phase-wedge-ode} is continuous. After a continuous extension
across the phase boundary, Peano's theorem gives a local $C^1$
solution $\ell$ with $\ell(S_{\min})=R_{\max}$.
The shock inequalities $\lambda_-<\tan\theta<\phi'<\lambda_+$,
Lemma~\ref{lem:initialsignsphase}, and \eqref{eq:phase-wedge-ode} give
\[
\sigma'(S_{\min})<\ell'(S_{\min})<0.
\]
Since $\ell(S_{\min})=\sigma(S_{\min})=R_{\max}$, choosing
$S_0>S_{\min}$ sufficiently close to $S_{\min}$ ensures
$\sigma<\ell<R_{\max}$ and $\ell'<0$ on $(S_{\min},S_0]$,
and $x_R<0<x_S$ on $\overline{\mathcal D_{S_0}}$.

To prove uniqueness, use $x_R<0$ to introduce the map
\[
(R,S)\longmapsto(S,X)=(S,x(R,S)),
\qquad R=\Phi(S,X),
\]
where $\Phi$ is its $C^1$ inverse in the $R$ variable near $O^*$.
Along a solution, $X(S)=x(\ell(S),S)$ satisfies
$X'=x_R\ell'+x_S$, so \eqref{eq:phase-wedge-ode} becomes
\[
X'=\mathcal G(S,X),\qquad
\mathcal G(S,X):=\left.
\frac{\lambda_+-\lambda_-}{\lambda_+-\tan\theta}x_S
\right|_{R=\Phi(S,X)}.
\]
Since $x_{SR}$ is continuous,
\[
\mathcal G_X
=\left.\frac{1}{x_R}\partial_R
\left(\frac{\lambda_+-\lambda_-}{\lambda_+-\tan\theta}x_S\right)
\right|_{R=\Phi(S,X)}\in C^0.
\]
After reducing $S_0$ if necessary, this derivative is bounded on the
image of $\mathcal P\cap\{S\le S_0\}$. Each section with $S$ fixed
is an interval, so $\mathcal G$ is Lipschitz in $X$ uniformly in $S$.
All solutions have $X(S_{\min})=x(O^*)=0$; Gr\"onwall's inequality
therefore gives uniqueness of $X$, and hence of $\ell=\Phi(S,X)$.
This proves $\mathsf H(S_0)$, including uniqueness.
\end{proof}

We investigate the monotonicity of entropy variables and density as follows.

\begin{proposition}\label{prop:entropy-density-monotonicity}
Assume \textup{(H1)--(H2)} and $\mathsf H(S_0)$. The entropy transported from the
shock is uniquely defined. Moreover, the flow satisfies
\begin{equation}\label{eq:09177}
    \Sigma_S<0,\qquad
\rho_S>0
\qquad\text{in }\mathcal D_{S_0}.
\end{equation}
\end{proposition}

\begin{proof}
Under $\mathsf H(S_0)$, \eqref{eq:entropyRS} becomes
\[
\Sigma_R+\mu\Sigma_S=0,
\qquad
\mu:=\frac{(\lambda_+-\tan\theta)x_R}
{(\tan\theta-\lambda_-)x_S}<0.
\]
Denote the corresponding characteristic through $(\tilde R,\tilde S)\in\mathcal D_{S_0}$
by $S=\mathscr S^{\tilde R,\tilde S}(R)$, where
\begin{equation}\label{eq:entropycharacteristics}
\frac{d}{dR}\mathscr S^{\tilde R,\tilde S}(R)
=\mu(R,\mathscr S^{\tilde R,\tilde S}(R)),\qquad
\mathscr S^{\tilde R,\tilde S}(\tilde R)=\tilde S.
\end{equation}
Note that $x_S>0$ allows the change of variable
$X=x(R,S)$, with local inverse $S=\Psi(R,X)$.
Along a characteristic, $X'=x_R+\mu x_S$, so
\[
X'=G(R,X),\qquad
G(R,X):=\left.
\frac{\lambda_+-\lambda_-}{\tan\theta-\lambda_-}x_R
\right|_{S=\Psi(R,X)}.
\]
The continuous mixed derivative furnished by
Lemma~\ref{lem:local-hodograph-cauchy} gives the continuous derivative
\[
G_X(R,X)=\left.\frac{1}{x_S}\partial_S
\left(\frac{\lambda_+-\lambda_-}{\tan\theta-\lambda_-}x_R\right)
\right|_{S=\Psi(R,X)}.
\]
Thus the transformed equation, and hence
\eqref{eq:entropycharacteristics}, has a unique local solution
depending $C^1$ on the initial point.

By Lemma~\ref{lem:initialsignsphase} and \eqref{eq:phase-wedge-ode},
\begin{equation}\label{eq:09171}
    \mu<\chi'<0\quad\text{on }\Gamma_{\rm sh},
\end{equation}
and
\begin{equation}\label{eq:09172}
    \omega'=\mu\quad\text{on }\Gamma_{\rm wg}.
\end{equation}
Thus each characteristic, traced toward increasing $R$, extends
uniquely to $\Gamma_{\rm sh}$. Denote the intersection by
$(r(\tilde R,\tilde S),\chi(r(\tilde R,\tilde S)))$. It holds
\begin{equation}\label{eq:09176}
    \Sigma(\tilde R,\tilde S)=\Sigma_{\rm sh}(\xi(\tilde R,\tilde S)),
\end{equation}
where $\xi(\tilde R,\tilde S):=\xi_{\rm sh}(r(\tilde R,\tilde S))$ and we use the same symbol $\Sigma$ in both coordinate systems.

To determine the derivative sign, fix $\tilde R$ and set
$X(R)=x(R,\mathscr S^{\tilde R,\tilde S}(R))$ and
$V=\partial_{\tilde S}X$. Differentiating the transformed equation gives
\[
V'=G_X(R,X(R))V,\qquad V(\tilde R)=x_S(\tilde R,\tilde S)>0.
\]
Since $V=x_S(R,\mathscr S^{\tilde R,\tilde S}(R))
\partial_{\tilde S}\mathscr S^{\tilde R,\tilde S}(R)$, it follows that
\begin{equation}\label{eq:09170}
\partial_{\tilde S}\mathscr S^{\tilde R,\tilde S}(R)
=\frac{x_S(\tilde R,\tilde S)}
{x_S(R,\mathscr S^{\tilde R,\tilde S}(R))}
\exp\!\left(\int_{\tilde R}^R G_X(t,X(t))\,dt\right)>0.
\end{equation}
Writing $r=r(\tilde R,\tilde S)$ and differentiating
$\mathscr S^{\tilde R,\tilde S}(r)=\chi(r)$ therefore yields
\begin{equation}\label{eq:SigmaStransport}
r_{\tilde S}
=\frac{\left.\partial_{\tilde S}\mathscr S^{\tilde R,\tilde S}(R)\right|_{R=r}}
{\chi'(r)-\mu(r,\chi(r))}>0.
\end{equation}
Together with \eqref{eq:monotonicityofRshSsh}, this gives
\begin{equation}\label{eq:09174}
\xi_{\tilde S}=\frac{r_{\tilde S}}{R'_{\rm sh}(\xi)}<0.
\end{equation}

On the other hand, the Rankine--Hugoniot relations give
\[
\frac{d}{dM_{n,\infty}^2}
\log\frac{p_{\rm sh}}{\rho_{\rm sh}^{\gamma}}
=
\frac{2\gamma(\gamma-1)(M_{n,\infty}^2-1)^2}
{[2\gamma M_{n,\infty}^2-(\gamma-1)]
M_{n,\infty}^2[(\gamma-1)M_{n,\infty}^2+2]}
>0,
\]
while
\[
\frac{d}{dx}M_{n,\infty}^2
=
\frac{2\phi'\phi''}
{\epsilon[1+(\phi')^2]^2}
>0.
\]
Therefore,
\begin{equation}\label{eq:09175}
    \frac{d}{dx}\Sigma_{\rm sh}(x)>0.
\end{equation}

Differentiating \eqref{eq:09176} with respect to $\tilde{S}$ and substituting \eqref{eq:09175} and \eqref{eq:09174}, we obtain
\begin{equation}\label{eq:09178}
    \Sigma_S(\tilde R,\tilde S)=\Sigma'_{\rm sh}(\xi)\xi_{\tilde S}<0,
\end{equation}
which proves the first inequality in \eqref{eq:09177}. It remains to prove $\rho_S>0$ in $\mathcal{D}_{S_0}$.

To this end, by \eqref{eq:phase-acoustic-variables}, we have
\begin{equation}
M_S=-\frac{1}{2\nu'(M)}<0,
\end{equation}
which together with $c^2=\frac{\gamma-1+2\epsilon}{2+(\gamma-1)M^2},$ derived from the Bernoulli law, gives
\begin{equation}\label{eq:cSpositive}
    c_S>0.
\end{equation}
Hence \eqref{eq:phase-thermodynamic-recovery}, together with \eqref{eq:09178} and \eqref{eq:cSpositive}, yields
\[
\frac{\rho_S}{\rho}
=\frac{1}{\gamma-1}
\left(\frac{2c_S}{c}-\frac{\Sigma_S}{\Sigma}\right)>0,
\]
which completes the proof.
\end{proof}

Define
\begin{equation}\label{eq:beta-definition}
\beta(R)
:=
\min_{\chi(R)\le S\le S_{\max}}
\lambda_-^2(R,S).
\end{equation}
Recall
\begin{equation}\label{eq:phase-mass-bound}
B_{\rm mass}(R):=
\frac{\rho_\infty\phi(x(R,\chi(R)))}
{(\rho c)(R,\chi(R))\sqrt{1+\beta(R)}}.
\end{equation}
We now estimate the distance between the shock and the wedge in the phase
plane along the $S$-direction.

\begin{proposition}[Shock-to-wedge distance estimate]\label{pro:distanceofshcokandwedge}
Assume \textup{(H1)--(H2)} and $\mathsf H(S_0)$. Then \begin{equation}\label{eq:upperbound}
x(R,S)-x(R,\chi(R))
\le
B_{\rm mass}(R), \quad (R,S)\in\overline{\mathcal D_{S_0}}.
\end{equation}

\end{proposition}

\begin{figure}[htbp]
\centering
\begin{tikzpicture}[
    x=1.02cm,y=0.91cm,
    >=Stealth,
    every node/.style={font=\small},
    guide/.style={densely dashed,gray!55,line width=0.4pt},
    shock/.style={black,line width=0.9pt},
    wedge/.style={blue!65!black,line width=0.9pt},
    cut/.style={orange!85!black,line width=1.15pt},
    point/.style={circle,fill=black,inner sep=1.35pt}
]
% The curves are drawn schematically as R = R(S), with both decreasing.
\begin{scope}
\pgfmathsetmacro{\Rcut}{4.3}
\pgfmathsetmacro{\Sshock}{(-0.6+sqrt(0.6*0.6+4*0.1375*(5.6-\Rcut)))/(2*0.1375)}
\pgfmathsetmacro{\Swedge}{(-0.25+sqrt(0.25*0.25+4*0.065*(5.6-\Rcut)))/(2*0.065)}
\fill[blue!9]
    plot[domain=0:\Sshock,samples=50,variable=\t]
        ({5.6-0.6*\t-0.1375*\t*\t},\t)
    -- (\Rcut,\Swedge)
    -- plot[domain=\Swedge:0,samples=50,variable=\t]
        ({5.6-0.25*\t-0.065*\t*\t},\t) -- cycle;
\draw[->,line width=0.65pt] (0.55,-0.2) -- (0.55,4.65) node[above] {$S$};
\draw[->,line width=0.65pt] (0.55,0) -- (6.15,0) node[right] {$R$};
\draw[guide] (0.55,4) -- (3.56,4);
\draw[guide] (3.56,0) -- (3.56,4);
\draw[guide] (\Rcut,0) -- (\Rcut,\Sshock);
\draw[shock] plot[domain=0:4,samples=80,variable=\t]
    ({5.6-0.6*\t-0.1375*\t*\t},\t);
\draw[wedge] plot[domain=0:4,samples=80,variable=\t]
    ({5.6-0.25*\t-0.065*\t*\t},\t);
\draw[cut] (\Rcut,\Sshock) -- (\Rcut,\Swedge);
\node[point] at (\Rcut,\Sshock) {};
\node[point] at (\Rcut,\Swedge) {};
\node[point] at (\Rcut,2.3) {};
\node[anchor=east] at (\Rcut-0.1,2.3) {$P$};
\node[point] at (5.6,0) {};
\node[anchor=east] at (\Rcut-0.09,\Sshock-0.03) {$A$};
\node[anchor=south west] at (\Rcut+0.04,\Swedge+0.06) {$B$};
\node[anchor=south west] at (5.61,0.08) {$O^*$};
\node[anchor=east] at (0.48,4) {$S_0$};
\node[anchor=east] at (0.48,0) {$S_{\min}$};
\node[below=5pt] at (3.56,0) {$\ell(S_0)$};
\node[below=5pt] at (\Rcut,0) {$R$};
\node[below=5pt] at (5.6,0) {$R_{\max}$};
\node[anchor=east] at (2.2,3.18) {$\Gamma_{\rm sh}$};
\node[anchor=west,text=blue!65!black] at (4.22,3.72) {$\Gamma_{\rm wg}$};
\node at (4.72,1.71) {$\mathcal G_R$};
\node[font=\small] at (3.28,-0.93) {(a) $R>\ell(S_0)$};
\end{scope}
\begin{scope}[xshift=7.6cm]
\pgfmathsetmacro{\Rcut}{2.5}
\pgfmathsetmacro{\Sshock}{(-0.6+sqrt(0.6*0.6+4*0.1375*(5.6-\Rcut)))/(2*0.1375)}
\fill[blue!9]
    plot[domain=0:\Sshock,samples=50,variable=\t]
        ({5.6-0.6*\t-0.1375*\t*\t},\t)
    -- (\Rcut,4) -- (3.56,4)
    -- plot[domain=4:0,samples=50,variable=\t]
        ({5.6-0.25*\t-0.065*\t*\t},\t) -- cycle;
\draw[->,line width=0.65pt] (0.55,-0.2) -- (0.55,4.65) node[above] {$S$};
\draw[->,line width=0.65pt] (0.55,0) -- (6.15,0) node[right] {$R$};
\draw[guide] (0.55,4) -- (\Rcut,4);
\draw[guide] (3.56,0) -- (3.56,4);
\draw[guide] (\Rcut,0) -- (\Rcut,\Sshock);
\draw[shock] plot[domain=0:4,samples=80,variable=\t]
    ({5.6-0.6*\t-0.1375*\t*\t},\t);
\draw[wedge] plot[domain=0:4,samples=80,variable=\t]
    ({5.6-0.25*\t-0.065*\t*\t},\t);
\draw[cut] (\Rcut,\Sshock) -- (\Rcut,4) -- (3.56,4);
\node[point] at (\Rcut,\Sshock) {};
\node[point] at (\Rcut,4) {};
\node[point] at (\Rcut,3.48) {};
\node[anchor=east] at (\Rcut-0.1,3.48) {$P$};
\node[point] at (3.56,4) {};
\node[point] at (5.6,0) {};
\node[anchor=east] at (\Rcut-0.09,\Sshock-0.03) {$A$};
\node[above=4pt] at (\Rcut,4) {$T$};
\node[above=4pt] at (3.56,4) {$W_0$};
\node[anchor=south west] at (5.61,0.08) {$O^*$};
\node[anchor=east] at (0.48,4) {$S_0$};
\node[anchor=east] at (0.48,0) {$S_{\min}$};
\node[below=5pt] at (\Rcut,0) {$R$};
\node[below=5pt] at (3.56,0) {$\ell(S_0)$};
\node[below=5pt] at (5.6,0) {$R_{\max}$};
\node[anchor=east] at (1.58,3.35) {$\Gamma_{\rm sh}$};
\node[anchor=west,text=blue!65!black] at (4.23,3.35) {$\Gamma_{\rm wg}$};
\node at (3.74,2.75) {$\mathcal G_R$};
\node[font=\small] at (3.28,-0.93) {(b) $R\le\ell(S_0)$};
\end{scope}
\end{tikzpicture}
\caption{Schematic integration regions $\mathcal G_R$ for the mass estimate.
In both panels, $A=(R,\chi(R))$.
In (a), the vertical segment ends at $B=(R,\omega(R))$ on the wedge.
In (b), it ends at $T=(R,S_0)$ and is closed by the horizontal segment
to $W_0=(\ell(S_0),S_0)$; when $R=\ell(S_0)$, $T=W_0$. The point $P=(R,S)$ lies on the vertical segment.}
\label{fig:mass-regions}
\end{figure}
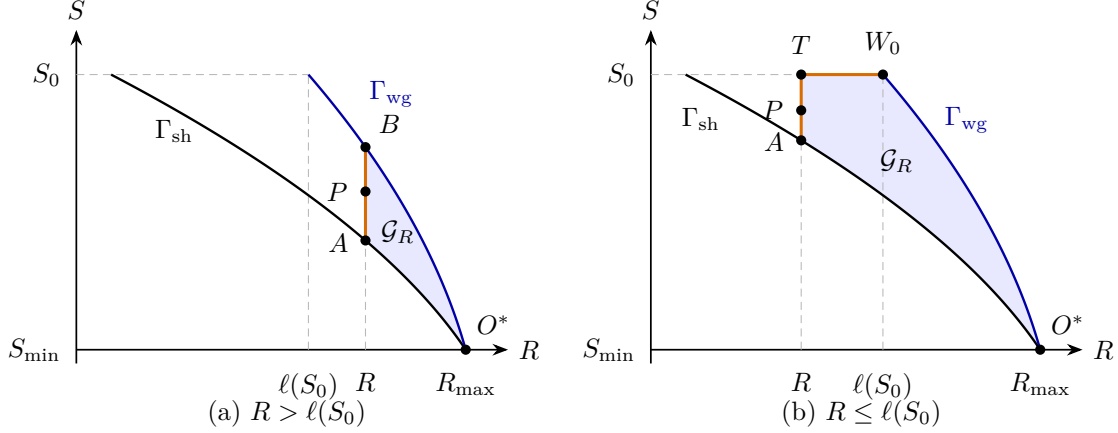

\begin{proof}
Fix $P=(R,S)\in\mathcal D_{S_0}$ and set $A=(R,\chi(R))$.
We integrate \eqref{eq:massRS} over the portion of
$\mathcal D_{S_0}$ to the right of the vertical line $r=R$,
denoted by $\mathcal G_R$ in Figure~\ref{fig:mass-regions}.
The upper endpoint of this vertical section depends on whether
$R\ge\ell(S_0)$ or $R<\ell(S_0)$.
In both cases, the wedge contribution vanishes by
\eqref{eq:presentationofWedgeinPhaseplane}, while the
Rankine--Hugoniot relation
\[
\rho_{\rm sh}(u_{\rm sh}\phi'-v_{\rm sh})=\rho_\infty\phi'
\]
converts the shock contribution into the incoming mass flux.

\textit{Case 1: $R\ge\ell(S_0)$.}
The vertical section ends at $B=(R,\omega(R))$ on the wedge,
with $\omega(R)\le S_0$. The integration region is bounded by the
shock arc $O^*A$, the vertical segment $AB$, and the wedge arc
$BO^*$, as shown in Figure~\ref{fig:mass-regions}(a). Hence
\begin{equation}\label{eq:massconservationinRS}
\int_{\chi(R)}^{\omega(R)}
\rho(v-u\lambda_-)x_S(R,\eta)\,d\eta
=
\rho_\infty\phi(x(R,\chi(R))).
\end{equation}

\textit{Case 2: $R<\ell(S_0)$.}
The vertical section ends at $T=(R,S_0)$ on the upper phase level $S=S_0$.
Writing $W_0=(\ell(S_0),S_0)$, the boundary now consists of
the shock arc $O^*A$, the vertical segment $AT$, the horizontal
segment $TW_0$, and the wedge arc $W_0O^*$; see
Figure~\ref{fig:mass-regions}(b). Integration gives
\begin{equation}\label{eq:mass-terminal-section}
\begin{aligned}
\int_{\chi(R)}^{S_0}
\rho(v-u\lambda_-)x_S(R,\eta)\,d\eta
&=\rho_\infty\phi(x(R,\chi(R)))\\
&\quad+\int_R^{\ell(S_0)}
\rho(u\lambda_+-v)x_R(r,S_0)\,dr\\
&\le\rho_\infty\phi(x(R,\chi(R))).
\end{aligned}
\end{equation}
Indeed, $u\lambda_+-v=u(\lambda_+-\tan\theta)>0$ and $x_R<0$,
so the horizontal contribution is negative when $R<\ell(S_0)$.

Since $\rho(v-u\lambda_-)x_S>0$, truncating either vertical
integral at the given phase level $S$ yields
\begin{equation}\label{eq:091721}
    \int_{\chi(R)}^S
\rho(v-u\lambda_-)x_S(R,\eta)\,d\eta
\le\rho_\infty\phi(x(R,\chi(R))).
\end{equation}
By $\lambda_-=\tan(\theta-\theta_m)$, we have
\begin{equation}\label{eq:091720}
    v-u\lambda_-
=
c\sqrt{1+\lambda_-^2}.
\end{equation}
By Proposition~\ref{prop:entropy-density-monotonicity},
\[
\rho_S>0,\qquad c_S>0,
\]
and $x_S>0$ under $\mathsf H(S_0)$. Substituting \eqref{eq:091720} into \eqref{eq:091721} and using the density monotonicity
and the definition of $\beta$ give
\begin{equation}\label{eq:mass-integral-bound}
\begin{aligned}
    B_{\rm mass}(R)&\ge\int_{\chi(R)}^S
\frac{c(R,\eta)}{c(R,\chi(R))}\,x_S(R,\eta)\,d\eta\\&\ge\int_{\chi(R)}^S
x_S(R,\eta)\,d\eta\\&=x(R,S)-x(R,\chi(R)).
\end{aligned}
\end{equation}
That is \[ x(R,S)-x(R,\chi(R))\le B_{\rm mass}(R), \quad (R,S)\in \mathcal D_{S_0}.\]
This inequality extends to $\overline{\mathcal D_{S_0}}$ by continuity
of $c$ and $x_S$. We complete the proof.
\end{proof}

Now we are ready to prove the main theorem.

\begin{proof}[Proof of Theorem~\ref{thm:main}]
Set
\[
S^*:=\sup\{S_0\in(S_{\min},S_{\max}]:\mathsf H(S_0)\ \text{holds}\}.
\]
By Lemma~\ref{lem:localH(S_0)}, this set is nonempty and
$S^*>S_{\min}$. Suppose, for contradiction, that $S^*<S_{\max}$.
Since $\ell$ is decreasing, it has a limit at $S^*$, and the
ordering \eqref{eq:09171} excludes a first return to $\Gamma_{\rm sh}$.

By Proposition~\ref{pro:distanceofshcokandwedge} and
\eqref{eq:maincriteria}, we have
\[
x(R,S)-x(R,\chi(R))\le B_{\rm mass}(R)<B_{\rm nd}(R)
\qquad\text{on }\overline{\mathcal D_{S^*}}.
\]
Thus \eqref{eq:xdistanceconditionphase} holds throughout this region, and
\[\mathcal D_{S^*}\subset\Omega^*_{S^*},\] which together with
Proposition~\ref{prop:nonvanishingxRxS} gives
\[
x_R<0<x_S\qquad\text{on }\overline{\mathcal D_{S^*}}.
\]
Then, by the continuity of $x_R,x_S$ from
Lemma~\ref{lem:local-hodograph-cauchy}, the same local construction
as in Lemma~\ref{lem:localH(S_0)} extends the wedge to some
$S_1\in(S^*,S_{\max}]$ with $\mathsf H(S_1)$, a contradiction.
Hence $S^*=S_{\max}$. That is,
$\mathsf H(S_{\max})$ holds.

Recall the inverse hodograph transformation
\[
F:(R,S)\longmapsto(x(R,S),y(R,S)).
\]
By $\mathsf H(S_{\max})$,  the Jacobian of the inverse hodograph transformation
\[
\det \mathrm DF=(\lambda_--\lambda_+)x_Rx_S>0
\]
on $\overline{\mathcal D_{S_{\max}}}$. Hence $F$ is locally invertible
in the connected region $\overline{\mathcal D_{S_{\max}}}$.

We next prove $F$ is globally invertible in $\overline{\mathcal D_{S_{\max}}}$. To this end, by Lemma \ref{lem:MonotonicityofRS} and \eqref{eq:presentationofWedgeinPhaseplane}--\eqref{eq:phase-wedge-ode}, we get that $x$
increases with $S$ along both the shock and the wedge.
Together with $x_R<0$, this makes each nonempty level set
$x=x_0$ a connected graph $R=R(S)$, along which
\[
\frac{dy}{dS}=(\lambda_--\lambda_+)x_S<0.
\]
Here \eqref{eq:hodographsystemphase} is applied. Thus $F$ is globally one-to-one. Set
\[
\Omega^*:=F(\mathcal D_{S_{\max}}),
\qquad \mathbf W^*:=F(\Gamma_{\rm wg}).
\]
The shock data \eqref{eq:phase-shock-data} give
$F(\Gamma_{\rm sh})=\mathbf S^*$, and the image of $S=S_{\max}$
is the terminal $\lambda_+$-characteristic. The inverse $F^{-1}$,
the entropy supplied by Proposition~\ref{prop:entropy-density-monotonicity},
and the equivalence in
Section~\ref{sec:phase-reformulation} therefore define the required
$C^1$ physical flow throughout $\Omega^*$. The wedge image is a graph
$y=f(x)$ satisfying $f'=v/u$, so $f\in C^2$.
Uniqueness follows from Lemma~\ref{lem:local-hodograph-cauchy},
Proposition~\ref{prop:entropy-density-monotonicity}, and uniqueness
of the physical wedge streamline. The proof is complete.
\end{proof}

\section{Numerical Reconstruction and Model Comparison}\label{sec:numerical}

The numerical examples use the full Euler Rankine--Hugoniot relations to
generate the post-shock Cauchy data. {The interior flow and wedge are then
reconstructed using the hybrid model \eqref{eq:PotantialFlow} in the
\((R,S)\)-plane.} We examine baseline reconstructions and grid convergence,
Mach-number dependence, and comparison with full-Euler inverse reconstructions. {We use the normalization}
\[
    \gamma=1.4,\qquad
    \rho_\infty=1,\qquad
    u_\infty=1,\qquad
    v_\infty=0,\qquad
    p_\infty=\frac{1}{\gamma M_\infty^2}.
\]
For Cases 1--3, \(X_*=1\). Each reconstruction is therefore specified by
\(M_\infty\) and the prescribed shock \(\phi\).

\subsection{Numerical reconstruction scheme}

Starting from the prescribed shock, the shock relations following \eqref{eq:shockcondition} provide the full Euler post-shock trace. The corresponding flow variables and Riemann variables \(R\) and \(S\) are then obtained from \eqref{eq:defofflowstates} and \eqref{eq:defofRiemanninvariants}, while the shock derivatives \(R'_{\rm sh}\) and \(S'_{\rm sh}\) are evaluated as in Lemma~\ref{lem:MonotonicityofRS}.

These quantities determine the Cauchy data \eqref{eq:xRxSinitialphase} for the hodograph equation \eqref{eq:hodographxequationphase}. Solving this equation in the \((R,S)\)-plane and then using the wedge equation \eqref{eq:presentationofWedgeinPhaseplane} yields the reconstructed wedge, whose physical image is \((x_{\rm wg}(S),y_{\rm wg}(S))\).

{
\noindent\textbf{Discretization.}
For Cases 1--3, we divide the shock interval \([0,1]\) into \(N\) uniform subintervals,
with nodes
\[
x_k=\frac{k}{N},\qquad 0\leq k\leq N,
\]
and evaluate \(\phi\), \(\phi'\), and \(\phi''\) at these nodes. The
triangular phase grid is
\[
(R_i,S_j),\qquad 0\leq i\leq j\leq N,
\]
where \(R_i=R_{\rm sh}(x_i)\) and \(S_j=S_{\rm sh}(x_j)\). On the diagonal
\(i=j\), the shock data give \(x_{ii}=x_i\) and \(y_{ii}=\phi(x_i)\), while
\(x_R\) and \(x_S\) are initialized by \eqref{eq:xRxSinitialphase}.

The derivatives entering \(a\) and \(b\) are \(\partial_R\lambda_-\) and \(\partial_S\lambda_+\), respectively, and are evaluated by centered differences with step size \(10^{-6}\).
The variables \(x_R\) and \(x_S\) are advanced from adjacent phase-grid nodes using an Euler predictor and trapezoidal corrector, with at most four fixed-point sweeps and a stopping tolerance of \(10^{-12}\).
The physical coordinates are recovered by trapezoidal integration in both phase directions, using \eqref{eq:hodographsystemphase} for \(y\). The two path values are averaged.
Interpolation along each fixed \(S\)-row and from the shock-family
data to the wedge is piecewise linear in \(R\).} The wedge equation is advanced in \(S\) using
an Euler predictor and trapezoidal corrector.

An absolute tolerance of \(10^{-10}\) is used for interpolation-range checks, whereas phase-grid range checks use \(10^{-9}\). Wedge-ODE denominators with absolute value below \(10^{-14}\) are rejected.

To assess the reconstruction, we compute over all phase-grid nodes
\[
    \min(-x_R),\qquad
    \min x_S,\qquad
    \min J_{xy},\qquad
    \min(M-1),
\]
where
\[
    J_{xy}:=\frac{\partial(x,y)}{\partial(R,S)}
    =
    (\lambda_- -\lambda_+)x_Rx_S.
\]
The normalized error in the wedge boundary condition is defined by
\[
    E_{\rm wg}
    :=
    \max_{\Gamma_{\rm wg}}
    \frac{|v-u f'(x)|}{\sqrt{u^2+v^2}}.
\]
{The wedge derivative is approximated by centered differences at interior
physical nodes and one-sided differences at the endpoints. Both endpoints
are included in \(E_{\rm wg}\).}
To compare the reconstructed shock-to-wedge distance with the mass bound
and non-degeneracy threshold, let
\[
    d_{\rm sw}(R):=x(R,\omega(R))-x(R,\chi(R)).
\]
Using the mass bound \(B_{\rm mass}\) and the non-degeneracy threshold
\(B_{\rm nd}\) expressed in \eqref{eq:phase-mass-bound} and
\eqref{eq:phase-nd-bound}, respectively, we compute
\[
    \eta_{\rm mass}(R):=\frac{d_{\rm sw}(R)}{B_{\rm mass}(R)},
    \qquad
    \eta_{\rm th}(R):=\frac{B_{\rm mass}(R)}{B_{\rm nd}(R)},
    \qquad
    \eta_{\rm nd}(R):=\frac{d_{\rm sw}(R)}{B_{\rm nd}(R)}.
\]
{The inequality \(\eta_{\rm th}<1\) represents the a priori
criterion \eqref{eq:maincriteria} and is computable from the shock data.
In contrast, \(\eta_{\rm mass}\) and \(\eta_{\rm nd}\) depend on the
reconstructed shock-to-wedge distance and are therefore a posteriori quantities.}

For each fixed-\(R_i\) family,
the integral entering the weight defining \(\alpha\) is approximated by
the trapezoidal rule, and {\(\alpha\) and \(\beta\) are evaluated by discrete
maximization and minimization, respectively.}
The maxima of these ratios {are taken over all points except} the common vertex.

\subsection{Baseline reconstructions and grid convergence}

Cases 1 and 2 use the prescribed shock profiles \(\phi_1\) and \(\phi_2\),
respectively:
\[
    \phi_1(x)=0.50(e^x-1),
    \qquad
    \phi_2(x)=0.22(e^{1.5x}-1),
    \qquad
    0\leq x\leq1.
\]
Both profiles are strictly convex, remain inside the admissible shock-slope
interval, and satisfy \(\mathcal D(\phi')<0\) for \(M_\infty=5\).
Table~\ref{tab:numerical-profiles} summarizes Cases 1--3.
For \(M_\infty=5\), Table~\ref{tab:shock-admissibility} lists the
shock-admissibility quantities for Cases 1 and 2.

\begin{table}[htbp]
\centering
\caption{Shock profiles used in the numerical examples.}
\label{tab:numerical-profiles}
\begin{tabular}{ccl>{\raggedright\arraybackslash}p{0.36\linewidth}}
\toprule
Case & \(M_\infty\) & Shock profile & Role \\
\midrule
1 & \(5\) & \(\phi_1(x)=0.50(e^x-1)\) & {baseline reconstruction and grid refinement} \\
2 & \(5\) & \(\phi_2(x)=0.22(e^{1.5x}-1)\) & {second baseline reconstruction} \\
3 & \(5,10,20,50\) & same as Case 1 & Mach-number dependence and full-Euler comparison \\
\bottomrule
\end{tabular}
\end{table}

\begin{table}[htbp]
\centering
\caption{Shock admissibility quantities for \(M_\infty=5\).}
\label{tab:shock-admissibility}
\begin{tabular}{crrrrr}
\toprule
Case & \(\min\phi'\) & \(\max\phi'\) & \(\min(M_{\rm sh}-1)\)
& \(\max \mathcal D(\phi')\) & \(\max R'_{\rm sh}\) \\
\midrule
1 & 0.5000 & 1.3591 & 0.5282 & \(-2.7824{\times}10^{-2}\) & \(-2.4416{\times}10^{-1}\) \\
2 & 0.3300 & 1.4790 & 0.4220 & \(-3.0305{\times}10^{-4}\) & \(-1.2537{\times}10^{-1}\) \\
\bottomrule
\end{tabular}
\end{table}

Case 1 provides the baseline admissible reconstruction, and
{Figure~\ref{fig:caseA-phase}} illustrates the associated phase-plane geometry.
Case 2 has a smaller initial shock slope than Case 1, but a larger
exponential rate and endpoint slope. Its physical-plane reconstruction is shown in
{Figure~\ref{fig:caseB-physical-phase}}.

\begin{figure}[htbp]
\centering
\includegraphics[width=0.62\textwidth]{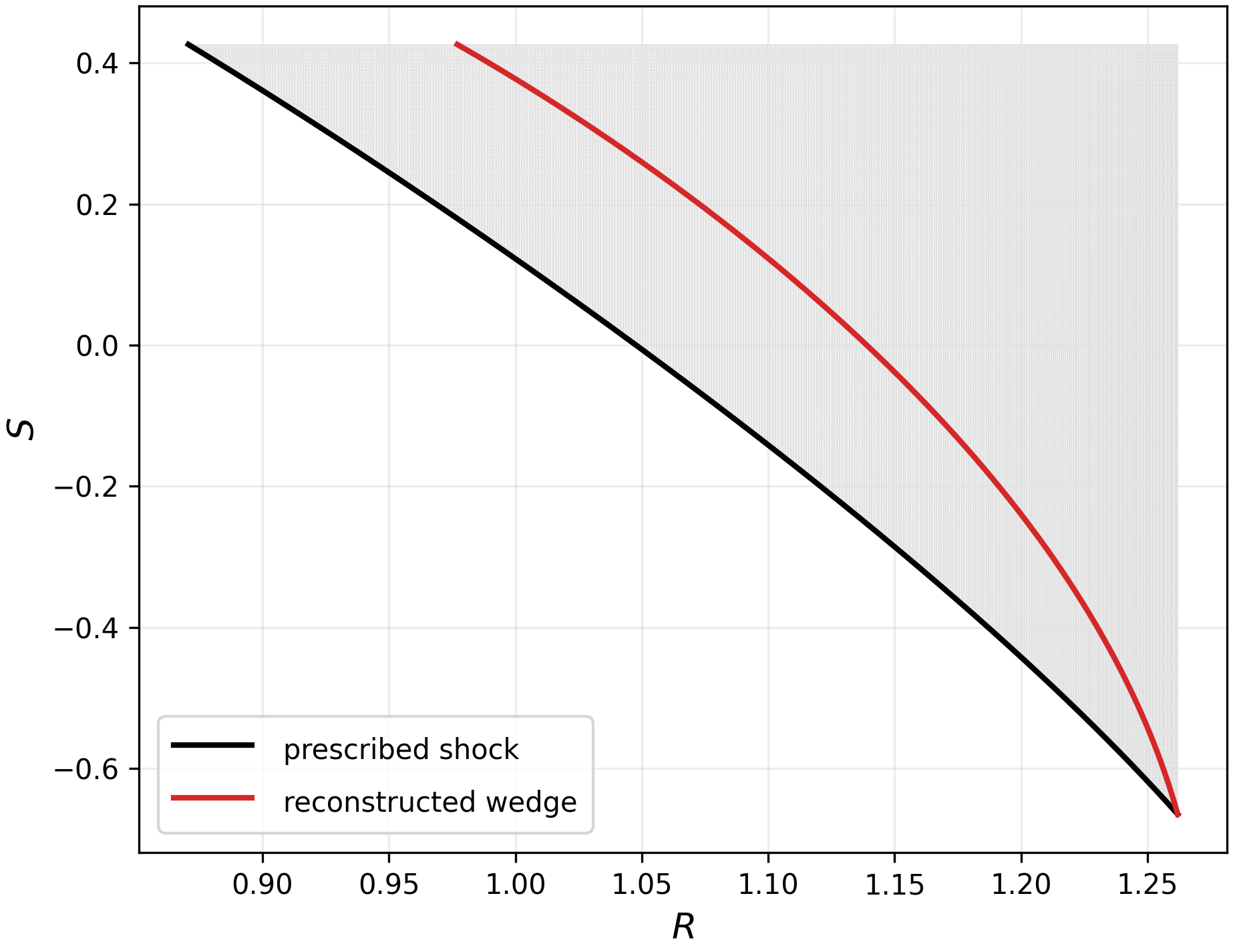}
\caption{Phase-plane geometry for Case 1: shock and wedge curves
\(\Gamma_{\rm sh}\) and \(\Gamma_{\rm wg}\) in the \((R,S)\)-plane.}
\label{fig:caseA-phase}
\end{figure}

\begin{figure}[htbp]
\centering
\includegraphics[width=0.62\textwidth]{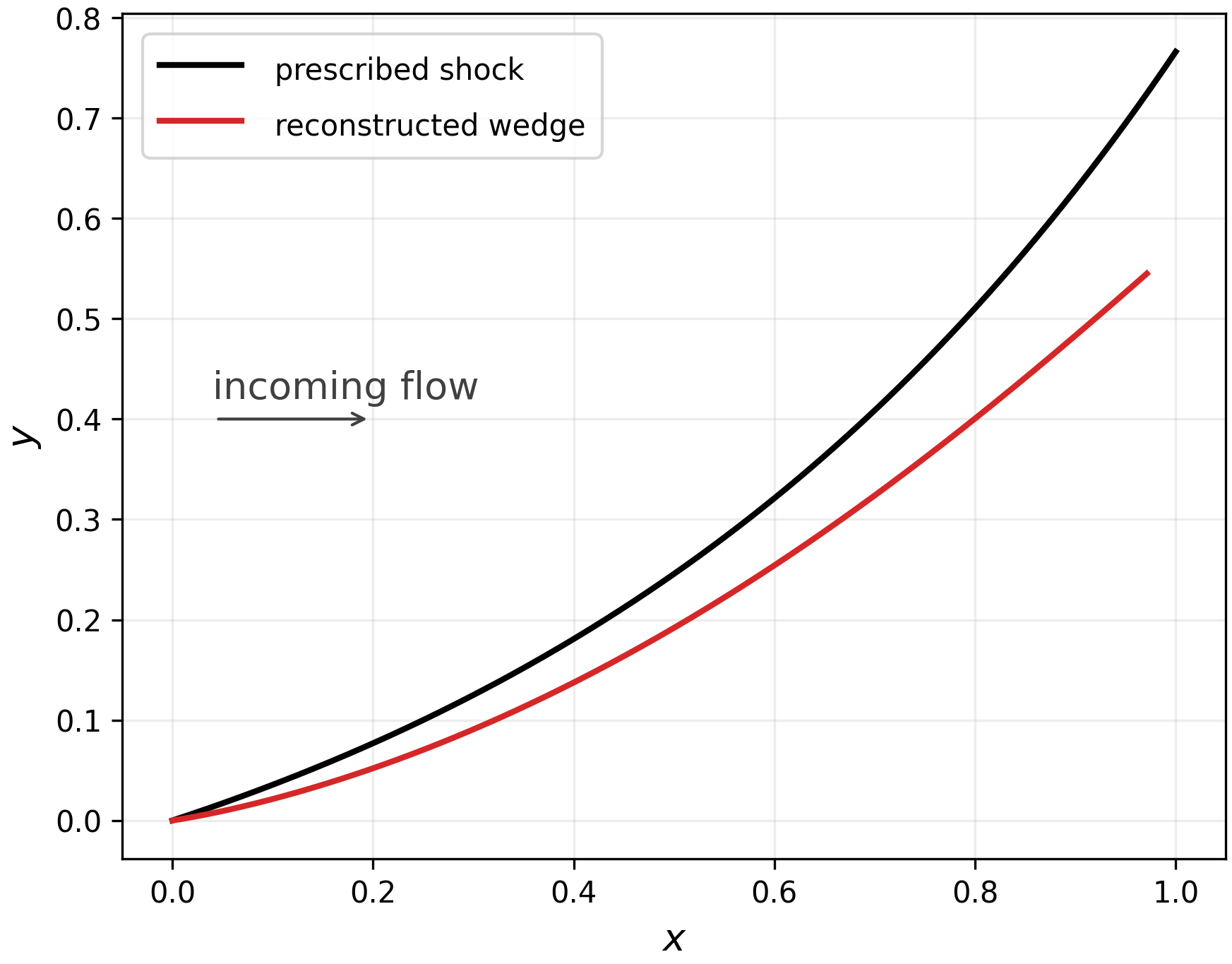}
\caption{Physical-plane reconstruction for Case 2.}
\label{fig:caseB-physical-phase}
\end{figure}

\begin{table}[htbp]
\centering
\caption{Reconstruction quantities for Cases 1 and 2 at \(N=400\).}
\label{tab:reconstruction-diagnostics}
{\setlength{\tabcolsep}{3.8pt}
\begin{tabular}{crrrrrrrr}
\toprule
Case & \(\min(-x_R)\) & \(\min x_S\) & \(\min J_{xy}\) & \(\min(M-1)\)
& \(\max\eta_{\rm mass}\) & \(\max\eta_{\rm th}\)
& {\(\max\eta_{\rm nd}\)} & \(E_{\rm wg}\) \\
\midrule
1 & 0.1857 & 0.2813 & 0.5327 & 0.5282 & 1.0002 & 0.1444
& {0.0751} & \(4.6426{\times}10^{-4}\) \\
2 & 0.0690 & 0.1582 & 0.2368 & 0.4220 & 0.9996 & 0.2002
& {0.0896} & \(5.2003{\times}10^{-4}\) \\
\bottomrule
\end{tabular}
}
\end{table}

The entries in Table~\ref{tab:reconstruction-diagnostics} are consistent with
the expected signs \(x_R<0\), \(x_S>0\), and \(J_{xy}>0\). The computed states
are also supersonic at all phase-grid nodes. {The small excess of
\(\max\eta_{\rm mass}\) above one in Case 1 is confined to the first
nonvertex point. A separate refinement check of \(\max\eta_{\rm mass}\)
at \(N=200,400,800\) shows that this excess decreases with mesh refinement.
The a posteriori non-degeneracy ratio \(\eta_{\rm nd}\), however,
remains well below one in both cases.}

\paragraph{Grid refinement for Case 1.}

We examine the grid convergence of the Case 1 reconstruction using
\[
    N=100,\quad 200,\quad 400,\quad 800.
\]
{An independently computed solution of the same scheme with
\(N_{\rm ref}=1600\) is used as the reference. To evaluate the errors,
the five wedge graphs are interpolated by a shape-preserving PCHIP
interpolant on their common physical interval \([x_a,x_b]\).
Both endpoints are retained, and the
\(L^2\) integral is evaluated by the composite trapezoidal rule.} We compute
\[
    E_{f,\infty}(N):=\|f_N-f_{\rm ref}\|_{L^\infty},
    \qquad
    E_{f,2}(N):=
    \left(
    \frac{1}{x_b-x_a}
    \int_{x_a}^{x_b}|f_N-f_{\rm ref}|^2\,{\rm d}x
    \right)^{1/2}.
\]
For either error quantity \(E\), the observed EOC is computed by
\[
    {\rm EOC}(N):=
    \frac{\log(E(N)/E(2N))}{\log 2}.
\]
The errors decrease under mesh refinement, and the observed EOCs are broadly
consistent with second-order {behavior} (Table~\ref{tab:convergence} and
{Figure~\ref{fig:convergence}}).
{This is an empirical convergence study rather than a proof of
second-order convergence.}

\begin{table}[htbp]
\centering
\caption{Grid convergence for Case 1 with \(N_{\rm ref}=1600\).}
\label{tab:convergence}
\begin{tabular}{crrrr}
\toprule
\(N\) & \(E_{f,\infty}\) & \({\rm EOC}_\infty\)
& \(E_{f,2}\) & \({\rm EOC}_2\) \\
\midrule
100 & \(1.2400{\times}10^{-5}\) & 1.7459 & \(6.8211{\times}10^{-6}\) & 1.9137 \\
200 & \(3.6971{\times}10^{-6}\) & 1.7401 & \(1.8104{\times}10^{-6}\) & 1.9424 \\
400 & \(1.1068{\times}10^{-6}\) & 2.1380 & \(4.7106{\times}10^{-7}\) & 2.3445 \\
800 & \(2.5145{\times}10^{-7}\) & -- & \(9.2752{\times}10^{-8}\) & -- \\
\bottomrule
\end{tabular}
\end{table}

\begin{figure}[htbp]
\centering
\includegraphics[width=0.62\textwidth]{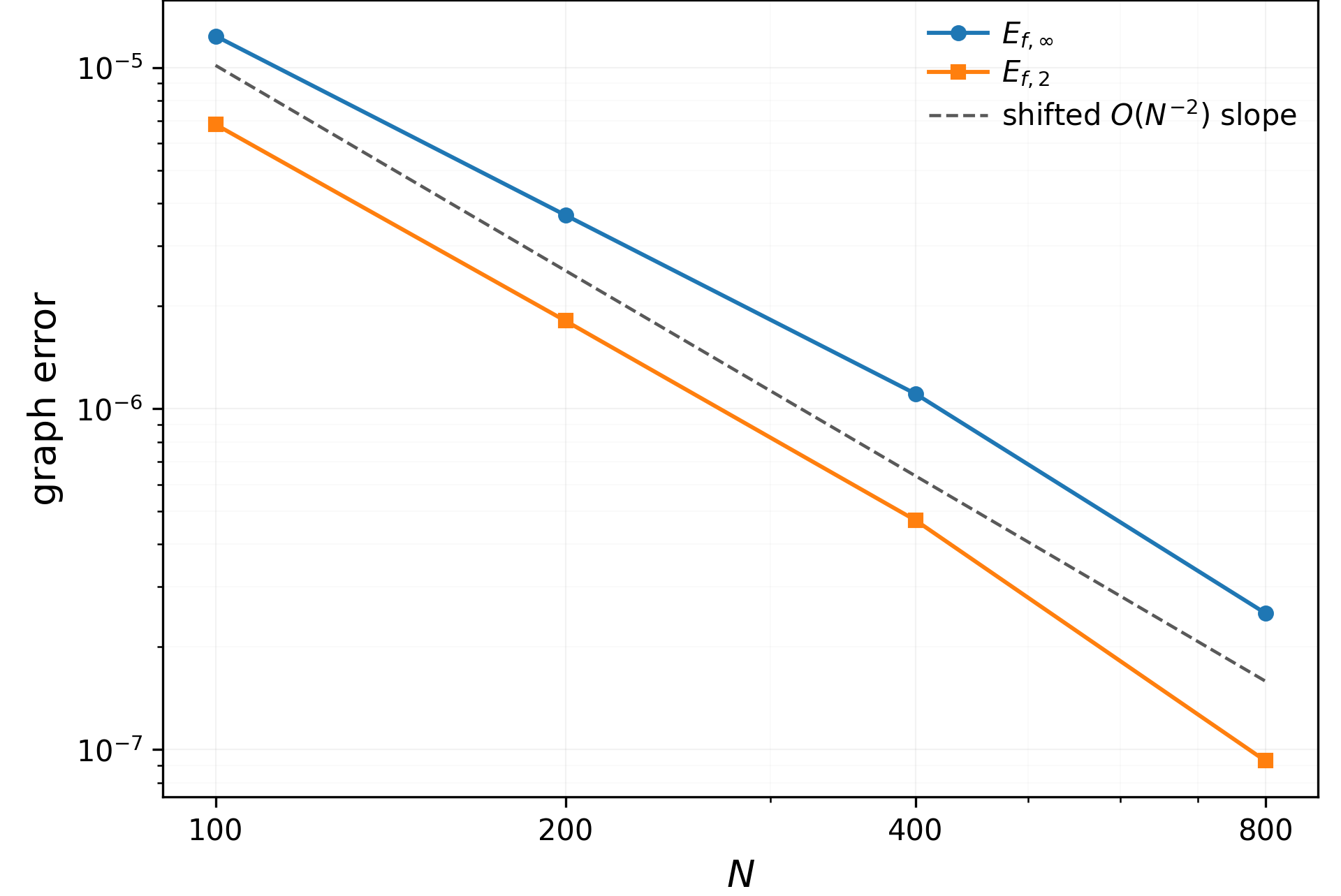}
\caption{Grid convergence for Case 1.}
\label{fig:convergence}
\end{figure}

\subsection{Mach-number dependence}
\label{sec:mach-number-dependence}

For Case 3, we fix the shock profile \(\phi_1\) from Case 1 and vary the
incoming Mach number:
\[
    M_\infty=5,\quad 10,\quad 20,\quad 50.
\]
The corresponding hybrid reconstructions with \(N=400\) are plotted in
{Figure~\ref{fig:mach-comparison}}.
To quantify the variation at higher Mach numbers, we compare the
\(M_\infty=10\) and \(20\) reconstructions with that for \(M_\infty=50\).
Denote by \(f_M\) the reconstructed wedge graph for incoming Mach
number \(M\). {After piecewise-linear interpolation over the pairwise
common physical intervals, we obtain}
\[
    \|f_{10}-f_{50}\|_{L^\infty}=1.8292\times10^{-2},
    \qquad
    \|f_{20}-f_{50}\|_{L^\infty}=3.9736\times10^{-3}.
\]
{Relative to the \(M_\infty=50\) reconstruction, the computed \(L^\infty\)
discrepancy is substantially smaller for \(M_\infty=20\) than for
\(M_\infty=10\).}

\begin{figure}[htbp]
\centering
\includegraphics[width=0.62\textwidth]{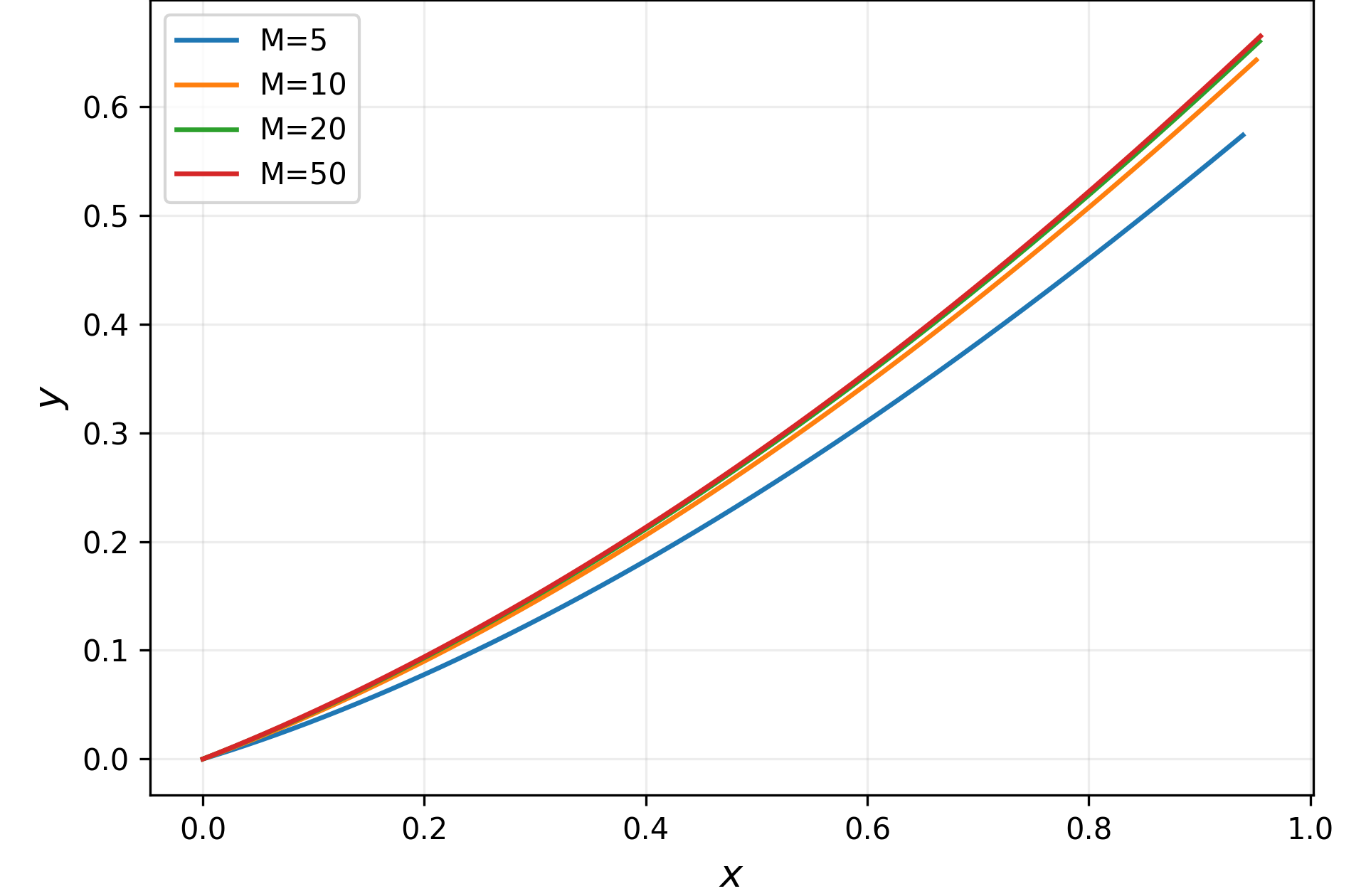}
\caption{Reconstructed wedges for Case 3 with \(M_\infty=5,10,20,50\).}
\label{fig:mach-comparison}
\end{figure}

\subsection{Comparison with full-Euler reconstruction}
\label{sec:hybrid-full-comparison}

{

For the prescribed shock \(\phi_1\) and the four incoming Mach numbers used
in Case 3, we compare the hybrid phase-plane reconstructions with full-Euler
inverse Cauchy reconstructions to quantify the discrepancy introduced by the
reduced model and to examine whether this discrepancy decreases as the
incoming Mach number increases. The two models use identical upstream data
at each Mach number. The full-Euler reconstruction uses the characteristic system
\eqref{eq:diagnalfullEuler}. In the full system, the wedge is recovered as a
material streamline up to its intersection with the terminal \(C_+\)
characteristic issuing from the endpoint of the prescribed shock.

Let \(f_{\rm hyb},f_{\rm full}\) denote the wedge graphs, with endpoint
abscissae \(X_{\rm hyb},X_{\rm full}\), and set
\(I_c=[0,\min(X_{\rm hyb},X_{\rm full})]\). We measure their relative
discrepancies by
\[
 E_\infty:=\frac{\|f_{\rm hyb}-f_{\rm full}\|_{L^\infty(I_c)}}
                    {\|f_{\rm full}\|_{L^\infty(I_c)}},\qquad
 E_2:=\frac{\|f_{\rm hyb}-f_{\rm full}\|_{L^2(I_c)}}
               {\|f_{\rm full}\|_{L^2(I_c)}},\qquad
 E_X:=\frac{|X_{\rm hyb}-X_{\rm full}|}{X_{\rm full}}.
\]
The first two quantities compare the wedge heights at the same physical
\(x\)-coordinate over the common interval \(I_c\), whereas \(E_X\)
measures the relative difference between the downstream endpoint abscissae.
The graphs are interpolated piecewise linearly on \(I_c\), and the
\(L^2\) integrals are evaluated by the trapezoidal rule.
The hybrid reconstructions use \(N=400\). The full-Euler
reconstructions use \(N=400\) at \(M_\infty=5,10\) and \(N=800\) at
\(M_\infty=20,50\). For the full-Euler reconstruction, we compute the
area-weighted RMS of the cellwise physical flux-balance defect divided by
cell area for mass, both momentum components, energy, and entropy transport.
The resolution is accepted when the largest of these five RMS residuals is
below \(10^{-2}\). Cells with area at most \(10^{-12}\) are excluded from
this calculation, and the excluded area fraction must remain below \(10^{-6}\).

For Case 1 at \(M_\infty=5\), the hybrid wedge lies above the
full-Euler wedge over most of the common physical interval \(I_c\),
with the largest separation at its downstream end
(Figure~\ref{fig:caseA-mach-ratio}). The maximum wall-height difference
on \(I_c\) is \(0.03704\). The hybrid reconstruction extends farther
downstream. The distinct endpoints reflect the different characteristic
evolution of the two models from the same terminal shock point.

\begin{figure}[htbp]
\centering
\includegraphics[width=0.62\textwidth]{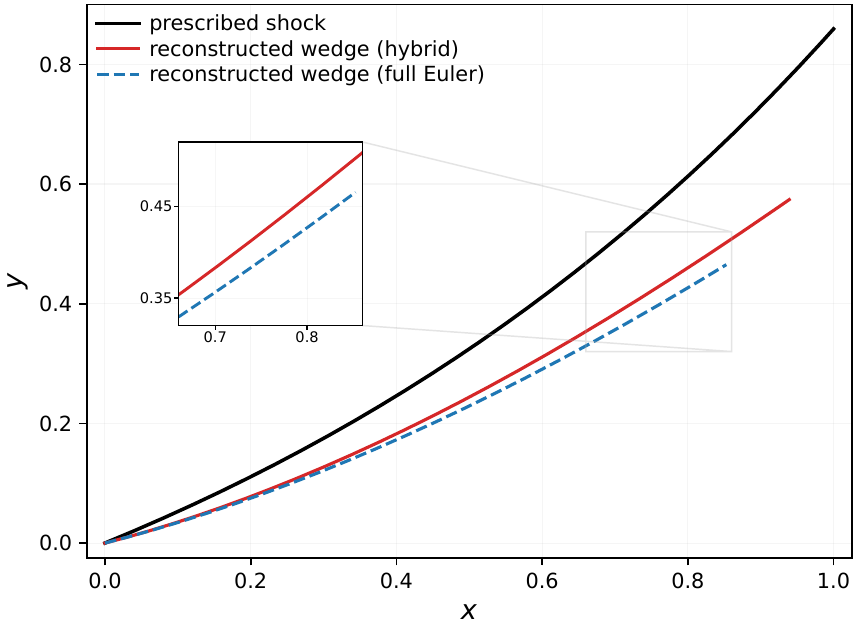}
\caption{Prescribed shock and hybrid/full-Euler reconstructed wedges for
Case 1 at \(M_\infty=5\). The inset enlarges their downstream
separation.}
\label{fig:caseA-mach-ratio}
\end{figure}

\begin{table}[htbp]
\centering
\caption{Hybrid/full-Euler discrepancies for \(\phi_1\).}
\label{tab:hybrid-full-mach}
{\setlength{\tabcolsep}{12pt}
\begin{tabular}{cccc}
\toprule
\(M_\infty\) & \(E_\infty\) & \(E_2\) & \(E_X\) \\
\midrule
5  & 0.079629 & 0.072480 & 0.101252 \\
10 & 0.081133 & 0.074386 & 0.100323 \\
20 & 0.079889 & 0.073273 & 0.099254 \\
50 & 0.079389 & 0.072803 & 0.098875 \\
\bottomrule
\end{tabular}
}
\end{table}

Across \(M_\infty=5\)--\(50\), the relative height discrepancies remain
close to \(8\%\) in \(L^\infty\) and \(7\%\) in \(L^2\), while the
endpoint discrepancy remains close to \(10\%\)
(Table~\ref{tab:hybrid-full-mach}). The weak variation of these discrepancies
over the sampled range indicates that the high-Mach stabilization observed
within the hybrid reconstructions in Case 3 does not, over this range,
correspond to convergence toward the full-Euler inverse reconstruction.
These are discrepancies between two inverse reconstructions, not an
independent full-Euler forward validation of the hybrid wedge.
}

{
\section{Conclusion}\label{sec:conclusion}

We have studied the inverse reconstruction of a wedge and its post-shock flow
from a prescribed leading shock and a uniform supersonic incoming state. The
full Euler Rankine--Hugoniot relations determine the post-shock Cauchy data,
while the interior reconstruction is governed by the reduced system considered
in this paper. Through the hodograph transformation, comparison of the
shock-to-wedge mass bound with the hodograph non-degeneracy threshold yields
an explicit sufficient criterion for preserving non-degeneracy throughout the
characteristic domain. Under the stated assumptions, this gives a global
one-to-one reconstruction with a $C^1$ post-shock flow and a $C^2$ wedge,
unique within the stated class of classical reconstructions.

The numerical results illustrate the reconstruction in both the phase and
physical planes and show empirical grid convergence with observed EOCs broadly
consistent with second-order behavior. Within the reduced model, the variation
among the reconstructed wedges decreases as the incoming Mach number increases.
This stabilization does not imply convergence toward the full-Euler
reconstruction. Indeed, comparisons based on the same prescribed shocks and
upstream data show a persistent and only weakly Mach-dependent discrepancy
over $M\_\infty=5$--$50$, thereby quantifying the effect of the reduced
interior model on the reconstructed wedge.

The present framework combines full Euler shock data with a reduced interior
system to obtain a mathematically tractable global reconstruction under an
explicit data-dependent criterion. A more systematic analysis of the
reduced/full-Euler modeling discrepancy, together with independent full-Euler
forward calculations for the reconstructed wedge, remains for future work.
}

\appendix

\section{Characteristic form of full Euler equations}\label{app:full-euler-characteristics}
\begin{lemma}
    For smooth supersonic full Euler states with the Bernoulli constant determined by the uniform incoming flow, the characteristic formulation is given by \cite{Courant1948,LiTsien1994}:
\begin{equation}\label{eq:diagnalfullEuler}
\left\{
\begin{aligned}
&R_x+\lambda_-R_y
=
\frac{\cos\theta_m}
{\gamma(\gamma-1)M^2\cos(\theta-\theta_m)}
\frac{\partial_n\Sigma}{\Sigma},
\\&
S_x+\lambda_+S_y
=
\frac{\cos\theta_m}
{\gamma(\gamma-1)M^2\cos(\theta+\theta_m)}
\frac{\partial_n\Sigma}{\Sigma},
\\&
\Sigma_x+\tan\theta\,\Sigma_y=0,
\\&
\frac{u^2+v^2}{2}
+\frac{\gamma p}{(\gamma-1)\rho}
=
\frac12+\frac{\gamma p_\infty}{(\gamma-1)\rho_\infty},
\end{aligned}
\right.
\end{equation}
where
\[
\partial_n:=-\sin\theta\,\partial_x+\cos\theta\,\partial_y
\]
is the derivative in the direction normal to the flow on its left.
Here, we use the notation introduced in
\eqref{eq:defofflowstates}--\eqref{eq:defofRiemanninvariants}.
\end{lemma}

\section{Mass conservation for the hybrid model}\label{app:mass-conservation}

\begin{lemma}[Mass conservation for the hybrid model]\label{lem:continuity-hybrid}
Let $(\rho,u,v,p)$ be a $C^1$ supersonic state whose associated variables $R,S,\Sigma$ satisfy \eqref{eq:PotantialFlow}. Then
\[
(\rho u)_x+(\rho v)_y=0.
\]
\end{lemma}

\begin{proof}
Write
\[
\partial_{\parallel}:=\cos\theta\,\partial_x+\sin\theta\,\partial_y,
\qquad
\partial_{\perp}:=-\sin\theta\,\partial_x+\cos\theta\,\partial_y.
\]
The Bernoulli relation and the definition of the Prandtl--Meyer function give
\[
dR=d\theta+\frac{\sqrt{M^2-1}}{q}\,dq,
\qquad
dS=d\theta-\frac{\sqrt{M^2-1}}{q}\,dq.
\]
Multiply the two acoustic equations in \eqref{eq:PotantialFlow} by
$\cos(\theta-\theta_m)$ and $\cos(\theta+\theta_m)$, respectively,
and subtract. The resulting identity is
\[
\partial_{\perp}\theta
=\frac{M^2-1}{q}\,\partial_{\parallel}q.
\]
Since $\partial_{\parallel}\Sigma=0$ and
$c^2=\gamma\Sigma\rho^{\gamma-1}$, differentiation along the flow gives
\[
\frac{2c}{\gamma-1}\,\partial_{\parallel}c
=\frac{c^2}{\rho}\,\partial_{\parallel}\rho.
\]
Combining this with the Bernoulli relation yields
$\partial_{\parallel}\rho=-(\rho q/c^2)\partial_{\parallel}q$.
Consequently,
\begin{align*}
(\rho u)_x+(\rho v)_y
&=q\partial_{\parallel}\rho+\rho\partial_{\parallel}q
  +\rho q\partial_{\perp}\theta\\
&=\rho(-M^2+1+M^2-1)\partial_{\parallel}q=0.
\end{align*}
\end{proof}

\section{Admissibility of the prescribed shock slope}\label{app:admissible-slopes}

\begin{lemma}[Admissible supersonic shock slopes]\label{lem:admissibleslope}
Let $1<\gamma<3$, $0<\epsilon<1$, and $s>0$.
For the incoming state \eqref{eq:initialdata}, the compressive full Euler
post-shock state is entropy admissible and supersonic if and only if
$s\in\mathcal I_{\rm sup}$, where $\mathcal I_{\rm sup}$ is defined by
\begin{equation}\label{eq:admissibleslope}
\begin{aligned}
\mathcal I_{\rm sup}&:=(s_-,s_+),
\qquad s_-:=\sqrt{\frac{\epsilon}{1-\epsilon}},\\
s_+&:=\left[
\frac{\mathscr L+\sqrt{\Delta}}
{2(\gamma-1+2\epsilon)(1-\epsilon)}-1
\right]^{1/2},
\end{aligned}
\end{equation}
where
\begin{equation}\label{eq:ConstL}
\mathscr L:=\gamma(3-\epsilon)+3\epsilon-1,
\qquad
\Delta:=\mathscr L^2
-8\gamma(\gamma-1+2\epsilon)(1-\epsilon).
\end{equation}
\end{lemma}

\begin{proof}
The incoming normal Mach number is
\[
M_{n,\infty}=\frac{s}{\sqrt{\epsilon(1+s^2)}}.
\]
For the nontrivial compressive branch of the full Euler shock relations,
the entropy condition is equivalent to $M_{n,\infty}>1$, that is,
$s>s_-$. The full Euler energy jump condition preserves total enthalpy,
so the post-shock state satisfies
\[
\frac{q_{\rm sh}^2}{2}+\frac{c_{\rm sh}^2}{\gamma-1}
=\frac12+\frac{\epsilon}{\gamma-1}.
\]
Combining this relation with the shock formulas following \eqref{eq:shockcondition} yields
\[
q_{\rm sh}^2-c_{\rm sh}^2
=\frac{P(s)}{s^2(\gamma+1)(1+s^2)},
\]
with $P$ as in \eqref{eq:polynomialP0}. Thus the post-shock state is
supersonic precisely when $P(s)>0$.

As a quadratic in $s^2$, $P$ has negative leading coefficient and
positive constant term. It therefore has one positive and one negative
root in that variable. The positive root is $s_+^2$, with $s_+$ given in
\eqref{eq:admissibleslope}; hence $P(s)>0$ for $s>0$ exactly when
$s<s_+$.
At $s=s_-$ the shock relations reduce to the incoming state,
so $q_{\rm sh}^2-c_{\rm sh}^2=1-\epsilon>0$.
In particular, $s_-<s_+$ and the interval is nonempty.
Combining the entropy and supersonicity conditions proves the result.
\end{proof}

\section{Shock-trace monotonicity and the sign regime}\label{app:sign-regime}

We prove the shock-trace property used to formulate the Cauchy data in
Section~2, and then record the location of the resulting sign regime.

\begin{proof}[Proof of Lemma~\ref{lem:MonotonicityofRS}]
Regard the post-shock variables as functions of $s=\phi'(x)$.
The velocity formulas following \eqref{eq:shockcondition} give
\begin{equation}\label{eq:explicitvelocitydif}
\frac{du_{\rm sh}}{ds}
=-\frac{4s}{(\gamma+1)(1+s^2)^2},
\qquad
\frac{dv_{\rm sh}}{ds}
=\frac{2[\epsilon(1+s^2)^2-s^4+s^2]}
{s^2(\gamma+1)(1+s^2)^2}.
\end{equation}
The polynomial
\begin{equation}\label{eq:polynomialP0}
P(s):=-(1-\epsilon)(\gamma-1+2\epsilon)s^4
+(4\epsilon^2+\epsilon\gamma-3\epsilon+\gamma+1)s^2
+2\epsilon^2
\end{equation}
satisfies
\[
q_{\rm sh}^2-c_{\rm sh}^2
=\frac{P(s)}{s^2(\gamma+1)(1+s^2)}.
\]
Consequently, $P(s)>0$ on $\mathcal I_{\rm sup}$.
Differentiating $q_{\rm sh}^2=u_{\rm sh}^2+v_{\rm sh}^2$ and
$\tan\theta_{\rm sh}=v_{\rm sh}/u_{\rm sh}$ yields
\begin{equation}\label{eq:explicitspeeddif}
q_{\rm sh}\frac{dq_{\rm sh}}{ds}
=-\frac{4B(s)}{s^3(\gamma+1)^2(1+s^2)^2},
\qquad
q_{\rm sh}^2\frac{d\theta_{\rm sh}}{ds}
=\frac{2N(s)}{s^2(\gamma+1)^2(1+s^2)^2},
\end{equation}
where the positive factors are
\[
B(s):=\epsilon^2(1+s^2)^2+\gamma s^4,
\qquad
N(s):=P(s)+\epsilon(\gamma+1)(1+s^2).
\]
Since $s=\phi'(x)$ and $\phi''(x)>0$, the signs
\begin{equation}\label{eq:monotonicityofspeed}
\frac{dq_{\rm sh}}{ds}<0,
\qquad
\frac{d\theta_{\rm sh}}{ds}>0
\end{equation}
yield the asserted monotonicity of $q_{\rm sh}$ and $\theta_{\rm sh}$
with respect to $x$. The Bernoulli relation and \eqref{eq:defofRiemanninvariants} give
\begin{equation}\label{eq:dotRS}
\begin{aligned}
S'_{\rm sh}
&=\phi''(x)\left(
\frac{d\theta_{\rm sh}}{ds}
-\frac{\sqrt{M_{\rm sh}^2-1}}{q_{\rm sh}}
\frac{dq_{\rm sh}}{ds}\right),\\
R'_{\rm sh}
&=\phi''(x)\left(
\frac{d\theta_{\rm sh}}{ds}
+\frac{\sqrt{M_{\rm sh}^2-1}}{q_{\rm sh}}
\frac{dq_{\rm sh}}{ds}\right).
\end{aligned}
\end{equation}
In particular, \eqref{eq:monotonicityofspeed} implies
\begin{equation}\label{eq:MonotonicityofS}
S'_{\rm sh}>0.
\end{equation}
For the other trace, substitution of \eqref{eq:explicitspeeddif} gives
\begin{equation}\label{eq:simplifieddotR}
R'_{\rm sh}
=\phi''(x)
\frac{2sN(s)-4\sqrt{M_{\rm sh}^2-1}\,B(s)}
{s^3(\gamma+1)^2(1+s^2)^2q_{\rm sh}^2}.
\end{equation}
The two terms in the numerator are positive before subtraction. Hence
\begin{equation}\label{eq:Longarrowequivalence}
\operatorname{sgn}R'_{\rm sh}
=\operatorname{sgn}\left[
s^2N(s)^2-4(M_{\rm sh}^2-1)B(s)^2\right].
\end{equation}
Finally, the full Euler post-shock formulas give
\begin{equation}\label{eq:SupersonicC}
M_{\rm sh}^2-1=\frac{(\gamma+1)P(s)}{C(s)},
\end{equation}
where
\[
C(s):=[2\epsilon+(2\epsilon+\gamma-1)s^2]
[2\gamma s^2-\epsilon(\gamma-1)(1+s^2)]>0.
\]
The positivity follows from $M_{n,\infty}>1$ and $\gamma>1$.
Multiplying the expression in \eqref{eq:Longarrowequivalence} by $C(s)$
therefore gives the sign polynomial
\begin{equation}\label{eq:Dtdefinition}
\mathcal D(s):=s^2C(s)N(s)^2
-4(\gamma+1)B(s)^2P(s),
\end{equation}
and proves the asserted sign identity.
\end{proof}

For \(\gamma=1.4\), Table~\ref{tab:sign-intervals} compares the admissible
positive shock-slope interval with the part where \(\mathcal D(s)>0\).
The endpoints follow from \eqref{eq:admissibleslope} and numerical root
finding for \eqref{eq:Dtdefinition}. The finite-Mach rows illustrate that
the sign change occurs close to the sonic endpoint. They are supplementary
shock-trace calculations, independent of the wedge reconstruction.

\begin{table}[htbp]
\centering
\caption{Admissible shock slopes and the region where
\(\mathcal D(s)>0\), for \(\gamma=1.4\). Finite endpoints are rounded
to six decimal places. The last row is the formal hypersonic limit.}
\label{tab:sign-intervals}
\begin{tabular}{ccc}
\toprule
\(M_\infty\) & \(\mathcal I_{\rm sup}\) & \(\mathcal D(s)>0\) \\
\midrule
\(5\) & \((0.204124,\,2.254928)\) & \((2.254456,\,2.254928)\) \\
\(2\) & \((0.577350,\,1.840650)\) & \((1.800046,\,1.840650)\) \\
\(+\infty\) & \((0,\,\sqrt6)\) & \(\varnothing\) \\
\bottomrule
\end{tabular}
\end{table}

For \(\epsilon=0\), the polynomial simplifies to
\[
\mathcal D(s)=-2\gamma s^{10}
\bigl(\gamma+1-(\gamma-1)s^2\bigr)
\bigl((\gamma+1)^2+(\gamma-1)^2s^2\bigr).
\]
It is strictly negative for
\(0<s<\sqrt{(\gamma+1)/(\gamma-1)}\). This algebraic limiting
observation does not extend the finite-Mach theorem to
\(\epsilon=0\), which lies outside its stated assumptions.

\section*{Acknowledgment}
Qianfeng Li was partially supported by Sino-German (CSC-DAAD) Postdoc Scholarship Program, 2023 (No. 57678375).

\bibliographystyle{plain}
\bibliography{Reference}
\end{document}